\documentclass{amsart}

\allowdisplaybreaks

\usepackage{ adjustbox }
\usepackage{ amsmath }
\usepackage{ amssymb }
\usepackage{ amsthm }
\usepackage{ array }
\usepackage{ booktabs }
\usepackage{ cite }
\usepackage[ usenames ]{ color }
\usepackage{ enumitem }
\usepackage{ latexsym }
\usepackage{ rotating }
\usepackage{ tikz }
\usepackage{ tikz-cd }
\usepackage{ url }
\usepackage[ all ]{ xy }

\usetikzlibrary{ trees }
\usetikzlibrary{ arrows }
\usetikzlibrary{ calc }
\usetikzlibrary{ positioning }
\usetikzlibrary{ decorations.pathmorphing }
\usetikzlibrary{ decorations.markings }

\tikzset{%
Square/.style = {%
inner sep = 0 pt, minimum width = 6 mm, minimum height = 6 mm, draw = black, fill = none, align = center
}}

\theoremstyle{definition}

\newtheorem{definition}{Definition}[section]
\newtheorem{algorithm} [definition]{Algorithm}
\newtheorem{example} [definition]{Example}
\newtheorem{notation} [definition]{Notation}
\newtheorem{remark} [definition]{Remark}

\theoremstyle{plain}

\newtheorem{lemma} [definition]{Lemma}
\newtheorem{proposition}[definition]{Proposition}
\newtheorem{theorem} [definition]{Theorem}

\renewcommand{\arg}{x}

\newcommand{\wedgehor}{{\,\scalebox{.67}{$\vartriangle$}\,}}
\newcommand{\wedgever}{{\,\scalebox{.67}{$\blacktriangle$}\,}}

\newcommand{\hor}{\Box}
\newcommand{\ver}{\blacksquare}

\newcommand{\HA}{\mathrm{A}_{\wedgehor}}
\newcommand{\VA}{\mathrm{A}_{\wedgever}}

\newcommand{\boxmid}{\,\adjustbox{valign=m}{\rotatebox{90}{$\boxminus$\;}}\,}

\makeatletter
\def\l@section {\@tocline{1}{0pt}{1pc}{} {}}
\def\l@subsection {\@tocline{2}{0pt}{1pc}{4.6em}{}}
\def\l@subsubsection{\@tocline{3}{0pt}{1pc}{7.6em}{}}
\renewcommand{\tocsection}[3]{%
\indentlabel{\@ifnotempty{#2}{\makebox[1.25em][l]{\ignorespaces#1#2.}}}#3}
\renewcommand{\tocsubsection}[3]{%
\indentlabel{\@ifnotempty{#2}{\hspace*{1.25em}\makebox[2.00em][l]{\ignorespaces#1#2.}}}#3}
\renewcommand{\tocsubsubsection}[3]{%
\indentlabel{\@ifnotempty{#2}{\hspace*{3.25em}\makebox[2.75em][l]{\ignorespaces#1#2.}}}#3}
\makeatother

\begin{document}


\title[COMMUTATIVITY IN DOUBLE INTERCHANGE SEMIGROUPS]
{COMMUTATIVITY IN DOUBLE INTERCHANGE SEMIGROUPS}


\author{Fatemeh Bagherzadeh}

\address{Department of Mathematics and Statistics, University of Saskatchewan, Canada}

\email{bagherzadeh@math.usask.ca}


\author{Murray Bremner}

\address{Department of Mathematics and Statistics, University of Saskatchewan, Canada}

\email{bremner@math.usask.ca}


\subjclass[2010]{%
Primary
20M50. 
Secondary
17A50, 
18D05, 
18D50, 
20M05, 
52C20. 
}


\keywords{Double semigroups,
interchange law,
associativity,
commutativity,
rectangular partitions,
double categories,
pasting diagrams.}


\thanks{%
The authors were supported by a Discovery Grant from NSERC,
the Natural Sciences and Engineering Research Council of Canada.
}


\begin{abstract}
We extend the work of Kock (2007) and Bremner \& Madariaga (2016) on commutativity in
double interchange semigroups (DIS) to relations with 10 arguments.
Our methods involve the free symmetric operad generated by two binary operations with no symmetry,
its quotient by the two associative laws,
its quotient by the interchange law,
and its quotient by all three laws.
We also consider the geometric realization of free double interchange magmas
by rectangular partitions of the unit square $I^2$.
We define morphisms between these operads which allow us to represent elements of free DIS
both algebraically as tree monomials and geometrically as rectangular partitions.
With these morphisms we reason diagrammatically about free DIS and prove our new commutativity relations.
\end{abstract}

\maketitle


{\footnotesize\tableofcontents}


\section{Introduction}

The primary motivation for this paper is the existence of unexpected commutativity properties
discovered during the last 10 years for double interchange semigroups.
Kock \cite[Proposition 2.3]{Kock2007} presents a $4 \times 4$ configuration
for which associativity and the interchange law imply the equality of two monomials,
with the same placement of parentheses and operation symbols,
but with different permutations of the arguments.
We display his result both algebraically and geometrically:
\begin{equation}
\label{Kockdiagram}
\begin{array}{c}
\begin{array}{l}
( a \,\hor\, b \,\hor\, c \,\hor\, d )
\,\ver\,
( e \,\hor\, f \,\hor\, g \,\hor\, h )
\,\ver\,
( i \,\hor\, j \,\hor\, k \,\hor\, \ell )
\,\ver\,
( m \,\hor\, n \,\hor\, p \,\hor\, q )
\equiv
\\
( a \,\hor\, b \,\hor\, c \,\hor\, d )
\,\ver\,
( e \,\hor\, g \,\hor\, f \,\hor\, h )
\,\ver\,
( i \,\hor\, j \,\hor\, k \,\hor\, \ell )
\,\ver\,
( m \,\hor\, n \,\hor\, p \,\hor\, q )
\end{array}
\\[5mm]
\begin{array}{c}
\begin{tikzpicture}[ draw = black, x = 6 mm, y = 6 mm ]
\node [Square] at ($(0, 0)$) {$a$};
\node [Square] at ($(1, 0)$) {$b$};
\node [Square] at ($(2, 0)$) {$c$};
\node [Square] at ($(3, 0)$) {$d$};
\node [Square] at ($(0,-1)$) {$e$};
\node [Square] at ($(1,-1)$) {$f$};
\node [Square] at ($(2,-1)$) {$g$};
\node [Square] at ($(3,-1)$) {$h$};
\node [Square] at ($(0,-2)$) {$i$};
\node [Square] at ($(1,-2)$) {$j$};
\node [Square] at ($(2,-2)$) {$k$};
\node [Square] at ($(3,-2)$) {$\ell$};
\node [Square] at ($(0,-3)$) {$m$};
\node [Square] at ($(1,-3)$) {$n$};
\node [Square] at ($(2,-3)$) {$p$};
\node [Square] at ($(3,-3)$) {$q$};
\end{tikzpicture}
\end{array}
\equiv
\begin{array}{c}
\begin{tikzpicture}[ draw = black, x = 6 mm, y = 6 mm ]
\node [Square] at ($(0, 0)$) {$a$};
\node [Square] at ($(1, 0)$) {$b$};
\node [Square] at ($(2, 0)$) {$c$};
\node [Square] at ($(3, 0)$) {$d$};
\node [Square] at ($(0,-1)$) {$e$};
\node [Square] at ($(1,-1)$) {$g$};
\node [Square] at ($(2,-1)$) {$f$};
\node [Square] at ($(3,-1)$) {$h$};
\node [Square] at ($(0,-2)$) {$i$};
\node [Square] at ($(1,-2)$) {$j$};
\node [Square] at ($(2,-2)$) {$k$};
\node [Square] at ($(3,-2)$) {$\ell$};
\node [Square] at ($(0,-3)$) {$m$};
\node [Square] at ($(1,-3)$) {$n$};
\node [Square] at ($(2,-3)$) {$p$};
\node [Square] at ($(3,-3)$) {$q$};
\end{tikzpicture}
\end{array}
\end{array}
\end{equation}
Note the transposition of $f$ and $g$.
We use the symbol $\equiv$ as an abbreviation for the statement that
the equation holds for all values of the arguments.
This interplay between algebra and geometry underlies all the results in this paper.
DeWolf \cite[Proposition 3.2.4]{DeWolf2013} used a similar argument with 10 variables
to prove that the operations coincide in every cancellative double interchange semigroup.
Bremner \& Madariaga used computer algebra to show that nine variables is the smallest number
for which such a commutativity property holds.
We display one of their results \cite[Theorem 4.1]{BM2016}; note the transposition of $e$ and $g$:
\begin{equation}
\label{BMdiagram}
\begin{array}{c}
\begin{array}{l}
( ( a \,\hor\, b ) \,\hor\, c )
\,\ver\,
( ( ( d \,\hor\, ( e \,\ver\, f ) ) \,\hor\, ( g \,\ver\, h ) ) \,\hor\, i )
\equiv
\\
( ( a \,\hor\, b ) \,\hor\, c )
\,\ver\,
( ( ( d \,\hor\, ( g \,\ver\, f ) ) \,\hor\, ( e \,\ver\, h ) ) \,\hor\, i )
\end{array}
\\[5mm]
\begin{array}{c}
\begin{tikzpicture}[ draw = black, x = 0.625 mm, y = 0.625 mm ]
\draw
( 0, 0) -- (48, 0)
( 0,24) -- (24,24)
( 0,48) -- (48,48)
( 6,36) -- (24,36)
( 0, 0) -- ( 0,48)
(12, 0) -- (12,48)
(24, 0) -- (24,48)
(48, 0) -- (48,48)
( 6,24) -- ( 6,48)
(24,24) -- (48,24)
( 6,12) node {$a$}
(18,12) node {$b$}
(36,12) node {$c$}
( 3,36) node {$d$}
( 9,30) node {$e$}
( 9,42) node {$f$}
(18,30) node {$g$}
(18,42) node {$h$}
(36,36) node {$i$};
\end{tikzpicture}
\end{array}
\equiv
\begin{array}{c}
\begin{tikzpicture}[ draw = black, x = 0.625 mm, y = 0.625 mm ]
\draw
( 0, 0) -- (48, 0)
( 0,24) -- (24,24)
( 0,48) -- (48,48)
( 6,36) -- (24,36)
( 0, 0) -- ( 0,48)
(12, 0) -- (12,48)
(24, 0) -- (24,48)
(48, 0) -- (48,48)
( 6,24) -- ( 6,48)
(24,24) -- (48,24)
( 6,12) node {$a$}
(18,12) node {$b$}
(36,12) node {$c$}
( 3,36) node {$d$}
(18,30) node {$e$}
( 9,42) node {$f$}
( 9,30) node {$g$}
(18,42) node {$h$}
(36,36) node {$i$};
\end{tikzpicture}
\end{array}
\end{array}
\end{equation}
Even though the operations are associative, we fully parenthesize monomials
so that the algebraic equation corresponds exactly with its geometric realization.

In this paper, we begin the classification of commutativity properties with 10 variables
which are not consequences of the known results with nine variables.

Following Loday \& Vallette \cite{LV2012}, we say \emph{algebraic operad}
to mean an operad in the symmetric monoidal category of vector spaces over a field $\mathbb{F}$.
When we say simply \emph{operad}, we mean a symmetric algebraic operad generated by
two binary operations.
For an earlier reference on operads and their applications,
see Markl, Shnider \& Stasheff \cite{MSS2002}.
For a more recent reference which emphasizes the algorithmic aspects,
see Bremner \& Dotsenko \cite{BD2016}.

The most common contemporary application of rectangular partitions is to
VLSI (very large scale integration): the process of producing integrated circuits
by the combination of thousands of transistors into a single silicon chip
\cite{KLMH2011}.
In microelectronics, block partitions are called \emph{floorplans}:
schematic representations of the placement of the major functional components of an integrated circuit.
Finding optimal floorplans subject to physical constraints leads to
NP-hard problems of combinatorial optimization.
An important subset consists of \emph{sliceable} floorplans;
these are similar to our dyadic partitions, except that we require the bisections to be exact.
Many NP-hard problems have polynomial time solutions in the sliceable case.
However, sliceable floorplans are defined to exclude the possibility that four subrectangles intersect 
in a point, thus rendering the interchange law irrelevant.


\section{Overview of results}

We recall basic definitions and results from the theory of algebraic operads.
We consider seven algebraic operads, each generated by two binary operations.


\subsection{Nonassociative operads}

The first four have nonassociative operations.

\begin{definition}
\label{freeoperad}
$\mathbf{Free}$
is the free operad generated by operations $\wedgehor$ and $\wedgever$:
we identify the basis monomials in arity $n$ with the set $\mathbb{B}_n$ of
all labelled rooted complete binary plane trees with $n$ leaves (\emph{tree monomials} for short),
where \emph{labelled} means that we assign an operation to each internal node (including the root),
and we choose a bijection between the leaf nodes and the argument symbols $\arg_1, \dots, \arg_n$.
For $n = 1$ we have only one tree, with no root and one leaf labelled $x_1$.
The partial compositions $\circ_i$ in this operad are defined as usual.
(We avoid using $\circ$ as an operation symbol since it conflicts with partial composition.)
\end{definition}

\begin{algorithm}
We recall in the general case the recursive algorithm for
\emph{converting a tree $T$ to a word $\mu(T)$}.
Let $\mathbf{O}$ be the free nonsymmetric operad generated by operations
$\Omega = \{ \, \omega_i \mid i \in I \, \}$ indexed by the set $I$,
with arities assigned by the function $a\colon I \to \mathbb{N}$.
The tree $\vert$ with one (leaf) node may be identified with the word $x_1$.
Each operation $\omega_i$ can be identified with either
(i) the word $\mu(T_i) = \omega_i(x_1,\dots,x_{a(i)})$, or
(ii) the planar rooted tree $T_i$ with root labelled $\omega_i$ and $a(i)$ leaves labelled
$x_1, \dots, x_{a(i)}$ from left to right.
This defines $\mu(T)$ for trees with exactly one internal node (the root).
Now let $T$ be a tree with at least two internal nodes (counting the root),
with root labelled $\omega_i$ and with $a(i)$ children (which are leaves or roots of subtrees)
denoted $T_1, \dots, T_{a(i)}$.
By induction we may assume that $\mu(T_1), \dots, \mu(T_{a(j)})$ have been defined,
and so we set $\mu( T ) = \omega_i( \mu(T_1), \dots, \mu( T_{a(i)} )$,
with the subscripts of the variables changed to produce the identity permutation.
By induction, this defines $\mu(T)$ for every basis tree $T$ in the free operad $\mathbf{O}$.
\end{algorithm}

\begin{definition}
\label{interoperad}
$\mathbf{Inter}$
is the quotient of $\mathbf{Free}$ by the operad ideal $\mathrm{I} = \langle \boxplus \rangle$
generated by the interchange law (also called exchange law, medial law, entropic law):
\begin{equation}
\label{intlaw}
\boxplus\colon\,
( a \,\wedgehor\, b ) \,\wedgever\, ( c \,\wedgehor\, d )
-
( a \,\wedgever\, c ) \,\wedgehor\, ( b \,\wedgever\, d )
\equiv 0.
\end{equation}
Example \ref{exampleinterchange} gives the geometrical explanation of our symbol for this relation.
\end{definition}

\begin{definition}
\label{bpoperad}
$\mathbf{BP}$ is the operad of \emph{block} (or \emph{rectangular}) partitions of the (open) unit
square $I^2$, $I = (0,1)$.
To be precise, a block partition $P$ of $I^2$ is determined by a finite set $C$ of open line segments
contained in $I^2$ such that $P = I^2 \setminus \bigcup C$
is the disjoint union of open subrectangles $(x_1,x_2) \times (y_1,y_2)$, called \emph{empty blocks}.
The segments in $C$, which are called \emph{cuts}, must be either horizontal or vertical,
that is $H = (x_1,x_2) \times \{y_0\}$ or $V = \{x_0\} \times (y_1,y_2)$,
where $0 \le x_0, x_1, x_2, y_0, y_1, y_2 \le 1$ with $x_1 < x_2$, $y_1 < y_2$.
We assume that the cuts are \emph{maximal} in the sense that if two elements $H, V \in C$
intersect then one is horizontal, the other is vertical, and $H \cap V$ is a single point.
$\mathbf{BP}$ has two binary operations:
the \emph{horizontal} (resp.~\emph{vertical}) operation $x \rightarrow y$ (resp.~$x \uparrow y$)
translates $y$ one unit to the east (resp.~north), forms the union of $x$ and translated $y$
to produce a block partition of a rectangle of width (resp.~height) two,
scales this rectangle horizontally (resp.~vertically) by one-half,
and produces another block partition.
The operadic analogues are as follows.
If $x$ is a block partition with $m$ parts ordered $x_1, \dots, x_m$ in some way,
then $x$ is an $m$-ary operation:
for any other block partition $y$ with $n$ parts, the partial composition
$x \circ_i y$ ($1 \le i \le m$) is the result of scaling $y$ to have the same size as $x_i$, and
replacing $x_i$ by the scaled partition $y$, producing a new block partition with $m{+}n{-}1$ parts.
Let $\boxmid$ and $\boxminus$ denote the two block partitions with two equal parts:
the first has a vertical (resp.~horizontal) cut and represents the horizontal (resp.~vertical) operation;
the parts are labelled 1, 2 in the positive direction, namely east (resp.~north).
These two operations form a basis for the homogeneous space $\mathbf{BP}(2)$.
The original operations are then defined as follows:
\begin{equation}
\label{bpoperations}
\begin{array}{l}
x \rightarrow y = ( \, \boxmid \circ_1 x \, ) \circ_{m+1} y = ( \, \boxmid \circ_2 y \, ) \circ_1 x,
\\[0.5mm]
x \uparrow y = ( \, \boxminus \circ_1 x \, ) \circ_{m+1} y = ( \, \boxminus \circ_2 y \, ) \circ_1 x.
\end{array}
\end{equation}
$\mathbf{BP}$ is a set operad, but we make it into an algebraic operad in the usual way
(see \S\ref{vectorsetoperads}): we define operations on elements and extend to linear combinations.
\end{definition}

\begin{definition}
\label{dbpoperad}
$\mathbf{DBP}$ is the unital suboperad of $\mathbf{BP}$ generated by $\mathbf{BP}(2)$,
where unital means we include the unary operation represented by $I^2$, the block partition with one part.
Thus $\mathbf{DBP}$ consists of the \emph{dyadic} partitions, meaning that every $P$ in $\mathbf{DBP}$ with $n{+}1$ parts
comes from some $Q$ in $\mathbf{DBP}$ with $n$ parts by exact bisection of a part of $Q$ horizontally or vertically.
The free double interchange magma is the algebra over $\mathbf{DBP}$ generated by $I^2$.
\end{definition}

\begin{algorithm}
\label{defbinarypartition}
For the general dimension $d \ge 1$, a \emph{dyadic block partition} $P$ with $k$ parts of the open unit
$d$-cube $I^d$, $I = (0,1)$, is constructed by setting $P_1 \leftarrow \{ I^d \}$ and performing the
following steps for $i = 1, \dots, k{-}1$:
\begin{itemize}[leftmargin=*]
\item
Choose an element $B \in P_i$ and a coordinate axis $j \in \{1,\dots,d\}$.
\item
Set $c \leftarrow \tfrac12(a_j{+}b_j)$ where $(a_j,b_j)$ is
the projection of $B$ onto the coordinate $x_j$.
\item
Set $\{ B', B'' \} \leftarrow B \setminus \{ \, x \in B \mid x_j = c \, \}$:
the disjoint open blocks obtained from bisecting $B$ by the hyperplane $x_j = c$.
\item
Set $P_{i+1} \leftarrow ( \, P_i \setminus \{ B \} \, ) \sqcup \{ B', B'' \}$:
in $P_i$, replace block $B$ with blocks $B'$, $B''$.
\end{itemize}
Finally, set $P \leftarrow P_k$.
\end{algorithm}

\begin{definition}
\label{defsubrectangle}
Let $P$ be a block partition of $I^2$ of arity $n$ determined by a set $C$ of line segments.
Then $P = I^2 \setminus C$ is the disjoint union of $n$ empty blocks $B_1, \dots, B_n$
and we indicate this by writing $P = \bigsqcup B_i$.
Suppose that the open rectangle $R = (x_1,x_2) \times ( y_1,y_2) \subseteq I^2$ admits a block
partition (in the obvious sense) into the disjoint union of a subset $B_{i_1}, \dots, B_{i_m}$ of
$m$ empty blocks from $P$.
In this case we say that $R$ is a \emph{subrectangle} of $P$ of arity $m$.
Every empty block $B_i$ is a subrectangle of $P$ of arity 1.
\end{definition}

\begin{definition}
\label{defgeomap}
The \emph{geometric realization map} $\Gamma\colon \mathbf{Free} \to \mathbf{BP}$
is the morphism of operads defined on tree monomials as follows:
$\Gamma( \,\vert\, ) = I^2$,
where $\vert$ is the (unique) tree with one node (a leaf),
and recursively we define
\begin{equation}
\begin{array}{c}
\Gamma( \, T_1 \wedgehor T_2 \, )
=
\adjustbox{valign=m}
{\begin{xy}
( 5, 5 )*+{\Gamma(T_1)};
( 15, 5 )*+{\Gamma(T_2)};
( 0, 0 ) = "1";
( 0, 10 ) = "2";
( 10, 0 ) = "3";
( 10, 10 ) = "4";
( 20, 0 ) = "5";
( 20, 10 ) = "6";
{ \ar@{-} "1"; "2" };
{ \ar@{-} "3"; "4" };
{ \ar@{-} "5"; "6" };
{ \ar@{-} "1"; "5" };
{ \ar@{-} "2"; "6" };
\end{xy}}
=
\Gamma(T_1) \rightarrow \Gamma(T_2),
\\[4mm]
\Gamma( \, T_1 \wedgever T_2 \, )
=
\adjustbox{valign=m}
{\begin{xy}
( 5, 5 )*+{\Gamma(T_1)};
( 5, 15 )*+{\Gamma(T_2)};
( 0, 0 ) = "1";
( 0, 10 ) = "2";
( 0, 20 ) = "3";
( 10, 0 ) = "4";
( 10, 10 ) = "5";
( 10, 20 ) = "6";
{ \ar@{-} "1"; "3" };
{ \ar@{-} "4"; "6" };
{ \ar@{-} "1"; "4" };
{ \ar@{-} "2"; "5" };
{ \ar@{-} "3"; "6" };
\end{xy}}
=
\Gamma(T_1) \uparrow \Gamma(T_2).
\end{array}
\end{equation}
\end{definition}

\begin{lemma}
\label{lemma1}
The image of $\Gamma$ is the operad $\Gamma( \mathbf{Free} ) = \mathbf{DBP}$ of dyadic block partitions.
The kernel of $\Gamma$ is the operad ideal $\mathrm{ker}(\Gamma) = \langle \boxplus \rangle$
generated by the interchange law.
Hence there is an operad isomorphism $\mathbf{Inter} \cong \mathbf{DBP}$.
\end{lemma}

\begin{proof}
The first statement is clear, the second is Lemma \ref{nasrinslemma},
and the third is an immediate consequence of the first and second.
\end{proof}

\begin{notation}
\label{gammanotation}
We write $\mathrm{I} = \mathrm{ker}(\Gamma) = \langle \boxplus \rangle$,
and $\gamma\colon \mathbf{Inter} \rightarrow \mathbf{DBP}$ for the isomorphism of Lemma \ref{lemma1}.
Then the geometric realization map $\Gamma = \iota \circ \gamma \circ \chi$ factors through
the natural surjection $\chi \colon \mathbf{Free} \twoheadrightarrow \mathbf{Inter}$
and the inclusion $\iota \colon \mathbf{DBP} \hookrightarrow \mathbf{BP}$:
\begin{equation}
\mathbf{Free}
\twoheadrightarrow
\mathbf{Free} / \mathrm{I}
=
\mathbf{Free} / \langle \boxplus \rangle
=
\mathbf{Inter}
\xrightarrow{\;\gamma\;}
\mathbf{DBP}
\hookrightarrow
\mathbf{BP}.
\end{equation}
See Figure \ref{bigpicture}.
\end{notation}

\begin{remark}
We mention but do not elaborate on the similarity between
(i)
the straightforward $n$-dimensional generalizations of the operads $\mathbf{BP}$ and $\mathbf{DBP}$,
and
(ii)
the much-studied operads $E_n$ which are weakly equivalent to the topological operads of little $n$-discs
and little $n$-cubes.
We refer the reader to McClure \& Smith \cite{MS2004} for further details and references.
\end{remark}


\subsection{Associative operads}

The last three operads have associative operations.

\begin{definition}
\label{assocboperad}
$\mathbf{AssocB}$
is the quotient of $\mathbf{Free}$ by the operad ideal $\mathrm{A} = \langle \HA, \VA \rangle$
generated by the associative laws for two operations:
\begin{equation}
\label{associativelaws}
\begin{array}{l}
\HA(a,b,c) = ( \, a \,\wedgehor\, b \, ) \,\wedgehor\, c - a \,\wedgehor\, ( \, b \,\wedgehor\, c \, ),
\\
\VA(a,b,c) = ( \, a \,\wedgever\, b \, ) \,\wedgever\, c - a \,\wedgever\, ( \, b \,\wedgever\, c \, ).
\end{array}
\end{equation}
This \emph{two-associative} operad is denoted $\mathbf{2as}$ by Loday \& Ronco \cite{LR2006}.
It is clumsy to regard the basis elements of $\mathbf{AssocB}$ as cosets of binary tree monomials
in $\mathbf{Free}$ modulo the ideal $\mathrm{A} = \langle \HA, \VA \rangle$.
We write $\overline{\wedgehor}$ and $\overline{\wedgever}$ for the operations in $\mathbf{AssocB}$,
where the bar indicates the quotient modulo $\mathrm{A}$.
To be precise, for tree monomials $x, y \in \mathbf{Free}$, we define
$\overline{\wedgehor}$ and $\overline{\wedgever}$ by these equations:
\begin{equation}
\begin{array}{l}
\left( \, x + \mathrm{A} \, \right)
\,\overline{\wedgehor}\,
\left( \, y + \mathrm{A} \, \right)
=
\left( \, x \,\wedgehor\, y \, \right) + \mathrm{A},
\\
\left( \, x + \mathrm{A} \, \right)
\,\overline{\wedgever}\,
\left( \, y + \mathrm{A} \, \right)
=
\left( \, x \,\wedgever y\, \, \right) + \mathrm{A}.
\end{array}
\end{equation}
\end{definition}

\begin{definition}
\label{assocnboperad}
$\mathbf{AssocNB}$ is an isomorphic copy of $\mathbf{AssocB}$ corresponding to
the following change of basis.
We write $\rho \colon \mathbf{AssocB} \to \mathbf{AssocNB}$ to represent rewriting
a coset representative (a binary tree) as a nonbinary (= not necessarily binary) tree.
The new basis consists of the disjoint union
$\{ x_1 \} \sqcup \mathbb{T}_\wedgehor \sqcup \mathbb{T}_\wedgever$
of the isolated leaf $x_1$ and two copies of the set $\mathbb{T}$
of all labelled rooted \emph{not necessarily binary} plane trees with
at least one internal node (counting the root).
We assume that each internal node has at least two children, and so
every tree in $\mathbb{T}$ has at least two leaves.
If $T$ is a tree in $\mathbb{T}$ with root $r$,
then the \emph{level} $\ell(s)$ of any internal node $s$
is the length of the unique path from $r$ to $s$ in $T$.
In $\mathbb{T}_\wedgehor$,
the root $r$ of every tree $T$ has label $\wedgehor$,
and the label of an internal node $s$ is $\wedgehor$ (resp.~$\wedgever$)
if $\ell(s)$ is even (resp.~odd).
In $\mathbb{T}_\wedgever$, the labels of the internal nodes are reversed.
If $T$ is in $\mathbb{T}_\wedgehor \sqcup \mathbb{T}_\wedgever$ (so $n \ge 2$),
then we include the $n!$ trees for all bijections between the leaves and
the argument symbols $\arg_1, \dots, \arg_n$.
Lemma \ref{assoclemma} gives a precise statement of the bijection between these two bases.
For further information, see Loday \& Ronco \cite[\S5]{LR2006}.
\end{definition}

\begin{remark}
If the choice of basis is not relevant, then we write $\mathbf{Assoc}$ to represent
the operad $\mathbf{AssocB} \cong \mathbf{AssocNB}$.
\end{remark}

\begin{definition}
\label{diaoperad}
$\mathbf{DIA}$
is the quotient of $\mathbf{Free}$ by the operad ideal $\langle \HA, \VA, \boxplus \rangle$.
This is the operad governing double interchange \emph{algebras},
which possess two associative operations satisfying the interchange law.
\end{definition}


\subsection{Set and vector operads}
\label{vectorsetoperads}

The operads $\mathbf{Inter}$, $\mathbf{AssocB}$, $\mathbf{AssocNB}$, $\mathbf{DIA}$
are defined by relations of the form $v_1 - v_2 \equiv 0$ (equivalently $v_1 \equiv v_2$)
where $v_1, v_2$ are cosets of tree monomials in $\mathbf{Free}$.
We could therefore work entirely with set operads, since we never need to consider linear combinations.

Vector spaces and sets are connected by a pair of adjoint functors:
the forgetful functor sending a vector space $V$ to its underlying set,
and the left adjoint sending a set $S$ to the free vector space on $S$
(the vector space with basis $S$).

The connection between vector spaces and sets is reflected in the relation between 
Gr\"obner bases for operads and the theory of rewriting systems:
if we compute a syzygy for two tree polynomials $v_1 - v_2$ and $w_1 - w_2$,
then the common multiple of the leading terms cancels, and we obtain another difference of tree monomials;
similarly, from a critical pair of rewrite rules $v_1 \mapsto v_2$ and $w_1 \mapsto w_2$,
we obtain another rewrite rule\footnote{We thank Vladimir Dotsenko for this clarification.}.

We state our main results in terms of set operads,
but strictly speaking, we work with algebraic operads.
A double interchange semigroup is a module over an operad $\mathbf{DIS}$ in the category of sets;
the corresponding notion over the algebraic operad $\mathbf{DIA}$ is a double interchange algebra.
Our main reason for using algebraic rather than set operads is that the former theory is
much better developed.
For more about the relation between set and algebraic operads, see Giraudo \cite[\S1.1.2]{Giraudo2016}.


\subsection{Morphisms between operads}
\label{morphisms}

Our goal in this paper is to understand the operad $\mathbf{DIA}$, which is the quotient of
$\mathbf{Free}$ by the operad ideal generated by the associative and interchange laws.
We have no convenient normal form for the basis monomials of $\mathbf{DIA}$ (that is,
no convenient way to choose a canonical representative for each equivalence class in
the quotient operad).
As we have just seen, there is a convenient normal form when we factor out associativity
but not interchange.
As we will see later (Lemma \ref{nasrinslemma}), there is also a convenient normal form
when we factor out interchange but not associativity: the dyadic block partitions.

Our approach will be to use the (tree) monomial basis of the operad $\mathbf{Free}$;
to these monomials, we apply rewriting rules which express associativity of each
operation (from right to left,  or the reverse) and the interchange law between the operations
(from black to white, or the reverse).
These rewritings convert one tree monomial in $\mathbf{Free}$ to another tree monomial which
is equivalent to the first modulo associativity and interchange.
In other words, given an element $X$ of $\mathbf{DIA}$ represented by a tree monomial
$T$ in $\mathbf{Free}$, these rewritings convert $T$ to another tree monomial $T'$ which
is in the same inverse image as $T$ with respect to the natural surjection
$\mathbf{Free} \twoheadrightarrow \mathbf{DIA}$.
We must allow undirected rewriting because of the complicated way
in which associativity and interchange interact: in order to pass from $T$ to $T'$,
we may need to apply associativity from left to right, then apply interchange, and then apply
associativity from right to left.

We present a commutative diagram of operads and morphisms in Figure \ref{bigpicture}:
\begin{itemize}[leftmargin=*]
\item
$\alpha$ is the natural surjection from $\mathbf{Free}$ onto $\mathbf{Assoc} = \mathbf{Free} / \mathrm{A}$.
\item
$\chi$ is the natural surjection from $\mathbf{Free}$ onto $\mathbf{Inter} = \mathbf{Free} / \mathrm{I}$.
\item
$\overline{\alpha}$ is the natural surjection from $\mathbf{Inter}$ onto
$\mathbf{DIA} = \mathbf{Inter} / \langle\HA{+}\mathrm{I},\VA{+}\mathrm{I}\rangle$.
\item
$\overline{\chi}$ is the natural surjection from $\mathbf{Assoc}$ onto
$\mathbf{DIA} = \mathbf{Assoc} / \langle\boxplus{+}\mathrm{A}\rangle$.
\item
$\overline{\chi} \circ \alpha = \overline{\alpha} \circ \chi$: the diagram commutes.
\item
For $\gamma$, $\iota$, $\Gamma = \iota \circ \gamma \circ \chi$
see Definition \ref{defgeomap} and Notation \ref{gammanotation}.
\item
For $\rho$ see Definition \ref{assocnboperad}.
\end{itemize}

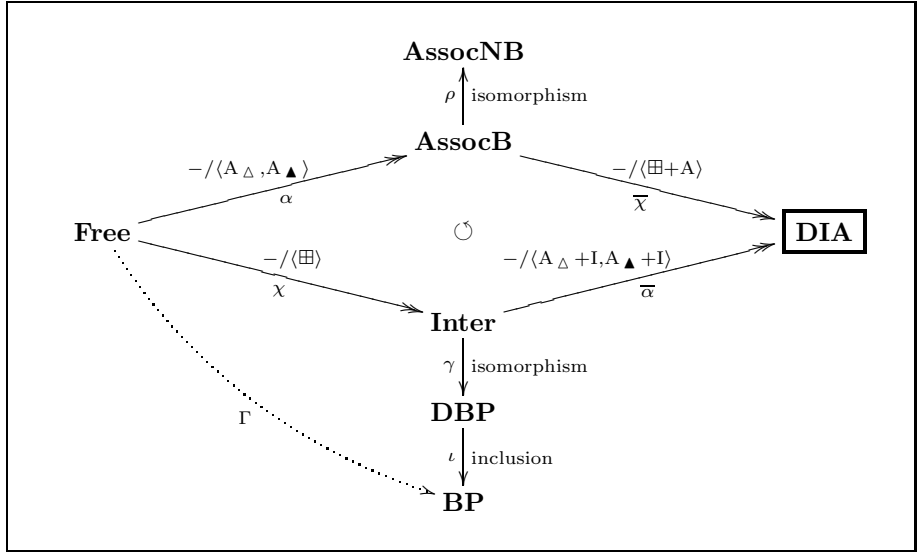
\begin{figure}[ht]
\begin{adjustbox}{center}
\setlength{\fboxsep}{12pt}
\fbox{
\setlength{\fboxsep}{4pt}
\setlength\fboxrule{1pt}
\begin{xy}
(  0,   0 )*+{\mathbf{Free}} = "free";
( 48, -36 )*+{\mathbf{BP}} = "bp";
( 48, -24 )*+{\mathbf{DBP}} = "dbp";
( 48, -12 )*+{\mathbf{Inter}} = "inter";
( 48,   0 )*+{\circlearrowleft} = "x";
( 48,  12 )*+{\mathbf{AssocB}} = "assocb";
( 48,  24 )*+{\mathbf{AssocNB}} = "assocnb";
( 96,   0 )*+{\boxed{\mathbf{DIA}}} = "dia";
{ \ar@{->>}^{-/\langle\HA,\VA\rangle\quad}_{\alpha} "free"; "assocb" };
{ \ar@{->>}^{-/\langle\boxplus\rangle}_{\chi} "free"; "inter" };
{ \ar@/_6mm/@{.>}_{\Gamma} "free"; "bp" };
{ \ar@{->}_{\gamma}^{\text{isomorphism}} "inter"; "dbp" };
{ \ar@{->>}^{-/\langle\HA{+}\mathrm{I},\VA{+}\mathrm{I}\rangle\qquad}_{\overline{\alpha}} "inter"; "dia" };
{ \ar@{->}^{\rho}_{\text{isomorphism}} "assocb"; "assocnb" };
{ \ar@{->>}^{-/\langle\boxplus{+}\mathrm{A}\rangle}_{\overline{\chi}} "assocb"; "dia" };
{ \ar@{->}^{\text{inclusion}}_{\iota} "dbp"; "bp" };
\end{xy}
}
\end{adjustbox}
\vspace{-10pt}
\caption{Big picture of operads and morphisms for rewriting monomials}
\label{bigpicture}
\end{figure}


\subsection{Diagram chasing and commutativity}
\label{diagramchasing}

By a \emph{monomial} $X$ in $\mathbf{DIA}$, we mean an equivalence class of (tree) monomials in $\mathbf{Free}$,
modulo the equivalence relation generated by the relations $\HA$, $\VA$, $\boxplus$.
Thus $X$ is a nonempty subset of (tree) monomials in $\mathbf{Free}$,
and a representative of $X$ is simply an element of $X$.

We start with a monomial $X$ in $\mathbf{DIA}$ and choose a convenient representative
(tree) monomial $T \in X$.
To the tree monomial $T$, we freely apply any sequence of rewrite rules of the following two types:
\begin{itemize}[leftmargin=*]
\item
\emph{Reassociating} in either direction (left to right or right to left)
with respect to either operation, $\wedgehor$ or $\wedgever$:
this means applying $\alpha$ to $T$ to obtain a unique element of $\mathbf{Assoc}$,
namely the coset $\alpha(T) = T + \mathrm{A}$;
rewriting the binary tree $T$ as a nonbinary tree as explained in Definition \ref{assocnboperad};
and applying $\alpha^{-1}$ by choosing a different binary tree $T'$ representing the same nonbinary tree:
$T + \mathrm{A} = T' + \mathrm{A}$, i.e., $\alpha(T) = \alpha(T')$.
\item
\emph{Interchanging} in either direction, left to right or right to left
(more precisely horizontal to vertical, or vertical to horizontal,
i.e., white root to black root, or black root to white root): this means
applying $\chi$ to $T$ to obtain a unique element of $\mathbf{Inter}$,
namely the coset $\chi(T) = T + \mathrm{I}$;
rewriting the binary tree $T$ as a dyadic block partition (Lemma \ref{nasrinslemma}:
interchange may only be applied in an unambiguous way to a \emph{binary} tree);
and applying $\chi^{-1}$ by choosing a different binary tree $T'$ representing the same dyadic block partition:
$T + \mathrm{I} = T' + \mathrm{I}$, i.e., $\chi(T) = \chi(T')$.
\end{itemize}
The role played by the nonbinary trees when reassociating is analogous to the role played by
the dyadic block partitions when interchanging:
the actual rewriting of the coset representatives (the tree monomials in $\mathbf{Free}$)
takes place using $\rho$ and $\rho^{-1}$ for associativity,
and $\gamma$ and $\gamma^{-1}$ for the interchange law.
We point out that:
\begin{itemize}[leftmargin=*]
\item
applying associativity, $T \mapsto T' \in \alpha^{-1}(\alpha(T))$,
changes the corresponding dyadic block partition,
but does not change the nonbinary tree monomial;
\item
applying the interchange law, $T \mapsto T' \in \chi^{-1}(\chi(T))$,
changes the corresponding nonbinary tree monomial,
but does not change the dyadic block partition.
\end{itemize}
This rewriting process is unavoidable because we do not have a well-defined normal form
for elements in $\mathbf{DIA}$, but we do have easily computable normal forms for elements of
$\mathbf{Free}$, $\mathbf{Assoc} = \mathbf{Free} / \langle \HA, \VA \rangle$ and
$\mathbf{Inter} = \mathbf{Free} / \langle \boxplus \rangle$.

We apply any number of these rewrite rules in any order,
and stop if and when we obtain a tree monomial $T''$ \emph{identical} to
the original monomial $T$ \emph{except} for the permutation of the arguments.
The equality in $\mathbf{DIA}$ of the cosets of the tree monomials $T$ and $T''$
in $\mathbf{Free}$ is a multilinear commutativity relation for double interchange algebras,
or equivalently for double interchange semigroups
(since we have been working exclusively with basis monomials).

\begin{figure}[ht]
\setlength{\fboxsep}{10pt}
\[
\boxed{
\begin{array}{l}
( ( a \wedgever b ) \wedgever c ) \wedgehor ( ( d \wedgehor e ) \wedgever f)
\;\; [\in \mathbf{Free}]
\\
\xrightarrow{\makebox[12mm]{$\alpha$}} \;
( ( a \wedgever b ) \wedgever c ) \wedgehor ( ( d \wedgehor e ) \wedgever f) + \mathrm{A}
\;\; [\in \mathbf{AssocB}]
\\
\xrightarrow{\makebox[12mm]{$\rho$}} \;
( a \wedgever b \wedgever c ) \wedgehor ( ( d \wedgehor e ) \wedgever f) + \mathrm{A}
\;\; [\in \mathbf{AssocNB}]
\\
\xrightarrow{\makebox[12mm]{$ \rho^{-1} $}} \;
( a \wedgever ( b \wedgever c ) ) \wedgehor ( ( d \wedgehor e ) \wedgever f) + \mathrm{A}
\;\; [\in \mathbf{AssocB}]
\\
\xrightarrow{\makebox[12mm]{$\alpha^{-1}$}} \;
( a \wedgever ( b \wedgever c ) ) \wedgehor ( ( d \wedgehor e ) \wedgever f)
\;\; [\in \mathbf{Free}]
\\
\xrightarrow{\makebox[12mm]{$\chi$}} \;
( a \wedgever ( b \wedgever c ) ) \wedgehor ( ( d \wedgehor e ) \wedgever f) + \mathrm{I}
\;\; [\in \mathbf{Inter}]
\\[8pt]
\xrightarrow{\makebox[12mm]{$\gamma$}} \;
\adjustbox{valign=m}
{\begin{xy}
(  0,   0 ) = "1";
(  0,  12 ) = "2";
(  0,  24 ) = "3";
( 12,   0 ) = "4";
( 12,  12 ) = "5";
( 12,  24 ) = "6";
( 24,   0 ) = "7";
( 24,  12 ) = "8";
( 24,  24 ) = "9";
(  0,  18 ) = "10";
( 12,  18 ) = "11";
( 18,   0 ) = "12";
( 18,  12 ) = "13";
{ \ar@{-} "1"; "3" };
{ \ar@{=} "4"; "6" };
{ \ar@{-} "7"; "9" };
{ \ar@{-} "1"; "7" };
{ \ar@{-} "2"; "8" };
{ \ar@{-} "3"; "9" };
{ \ar@{-} "10"; "11" };
{ \ar@{-} "12"; "13" };
(  6,   6 )*+{a} = "a";
(  6,  15 )*+{b} = "b";
(  6,  21 )*+{c} = "c";
( 15,   6 )*+{d} = "d";
( 21,   6 )*+{e} = "e";
( 18,  18 )*+{f} = "f";
\end{xy}}
=
\adjustbox{valign=m}
{\begin{xy}
(  0,   0 ) = "1";
(  0,  12 ) = "2";
(  0,  24 ) = "3";
( 12,   0 ) = "4";
( 12,  12 ) = "5";
( 12,  24 ) = "6";
( 24,   0 ) = "7";
( 24,  12 ) = "8";
( 24,  24 ) = "9";
(  0,  18 ) = "10";
( 12,  18 ) = "11";
( 18,   0 ) = "12";
( 18,  12 ) = "13";
{ \ar@{-} "1"; "3" };
{ \ar@{-} "4"; "6" };
{ \ar@{-} "7"; "9" };
{ \ar@{-} "1"; "7" };
{ \ar@{=} "2"; "8" };
{ \ar@{-} "3"; "9" };
{ \ar@{-} "10"; "11" };
{ \ar@{-} "12"; "13" };
(  6,   6 )*+{a} = "a";
(  6,  15 )*+{b} = "b";
(  6,  21 )*+{c} = "c";
( 15,   6 )*+{d} = "d";
( 21,   6 )*+{e} = "e";
( 18,  18 )*+{f} = "f";
\end{xy}}
\;\; [\in \mathbf{DBP}] \;\;
\left( \!\! \begin{tabular}{c} double line \\ denotes root \\ operation \end{tabular} \!\! \right)
\\
\xrightarrow{\makebox[12mm]{$\gamma^{-1}$}} \;
( a \wedgehor ( d \wedgehor e ) ) \wedgever ( ( b \wedgever c ) \wedgehor f) + \mathrm{I}
\;\; [\in \mathbf{Inter}]
\\
\xrightarrow{\makebox[12mm]{$\chi^{-1}$}} \;
( a \wedgehor ( d \wedgehor e ) ) \wedgever ( ( b \wedgever c ) \wedgehor f)
\;\; [\in \mathbf{Free}]
\\
\xrightarrow{\makebox[12mm]{$\alpha$}} \;
( a \wedgehor ( d \wedgehor e ) ) \wedgever ( ( b \wedgever c ) \wedgehor f) + \mathrm{A}
\;\; [\in \mathbf{AssocB}]
\\
\xrightarrow{\makebox[12mm]{$ \rho $}} \;
( a \wedgehor d \wedgehor e ) \wedgever ( ( b \wedgever c ) \wedgehor f) + \mathrm{A}
\;\; [\in \mathbf{AssocNB}]
\\
\xrightarrow{\makebox[12mm]{$ \rho^{-1} $}} \;
( ( a \wedgehor d ) \wedgehor e ) \wedgever ( ( b \wedgever c ) \wedgehor f) + \mathrm{A}
\;\; [\in \mathbf{AssocB}]
\\
\xrightarrow{\makebox[12mm]{$\alpha^{-1}$}} \;
( ( a \wedgehor d ) \wedgehor e ) \wedgever ( ( b \wedgever c ) \wedgehor f)
\;\; [\in \mathbf{Free}]
\end{array}
}
\]
\vspace{-15pt}
\caption{Example of rewriting in free double interchange semigroups}
\label{rewritingexample}
\end{figure}

\begin{example}
In Figure \ref{rewritingexample} we present a simple exercise in rewriting,
which does not lead to a commutativity property (such an example would require at least nine variables),
but which should suffice to illustrate the preceding discussion.
For clarity, we include in square brackets the position in Figure \ref{bigpicture} of the current object.
We emphasize that we never work directly with elements of $\mathbf{DIA}$.
\end{example}


\section{Background in categorical algebra}

Most of the operadic and geometric objects studied in this paper originate in category theory;
we mention Mac Lane \cite{MacLane1965}, Kelly \& Street \cite{KS1974}, Street \cite{Street1996}
as mathematical references, Kr\"omer \cite{Kromer2007} for historical and philosophical aspects.


\subsection{Many binary operations}

Many different structures may be regarded as extensions of the notion of semigroup to
the case of $d \ge 2$ binary operations.
(The results of this paper concern only $d = 2$.)
We give definitions for the general case $d \ge 2$ only when
this requires no more space than $d = 2$.

\begin{definition}
\label{defmtuplesemigroup}
A $d$-\emph{tuple magma} is a nonempty set $S$ with $d$ binary operations
$S \times S \to S$, denoted $(a,b) \mapsto a \star_i b$ for $1 \le i \le d$.
A $d$-\emph{tuple semigroup} is a $d$-tuple magma
in which every operation satisfies the associative law.
A $d$-\emph{tuple interchange magma} is a $d$-tuple magma
in which every pair of distinct operations satisfies the interchange law.
A $d$-\emph{tuple interchange semigroup} is a $d$-tuple semigroup
in which every pair of distinct operations satisfies the interchange law.
(Some authors refer to the last structure simply as ``a $d$-tuple semigroup''.)
\end{definition}

\begin{example}
\label{exampleinterchange}
Double interchange magmas have operations $\rightarrow$ (horizontal) and $\uparrow$ (vertical)
related by the interchange law, which expresses the equality of two sequences of bisections
which partition a square into four smaller squares:
\[
( a \rightarrow b ) \uparrow ( c \rightarrow d )
\, \equiv
\begin{array}{c}
\begin{tikzpicture}[draw=black, x=6 mm, y=6 mm]
\node [Square] at ($(0, 0.0)$) {$a$};
\node [Square] at ($(1, 0.0)$) {$b$};
\node [Square] at ($(0,-1.2)$) {$c$};
\node [Square] at ($(1,-1.2)$) {$d$};
\end{tikzpicture}
\end{array}
\equiv
\begin{array}{c}
\begin{tikzpicture}[draw=black, x=6 mm, y=6 mm]
\node [Square] at ($(0, 0)$) {$a$};
\node [Square] at ($(1, 0)$) {$b$};
\node [Square] at ($(0,-1)$) {$c$};
\node [Square] at ($(1,-1)$) {$d$};
\end{tikzpicture}
\end{array}
\equiv
\begin{array}{c}
\begin{tikzpicture}[draw=black, x=6 mm, y=6 mm]
\node [Square] at ($(0.0, 0)$) {$a$};
\node [Square] at ($(1.2, 0)$) {$b$};
\node [Square] at ($(0.0,-1)$) {$c$};
\node [Square] at ($(1.2,-1)$) {$d$};
\end{tikzpicture}
\end{array}
\equiv \;
( a \uparrow c ) \rightarrow ( b \uparrow d ).
\]
\end{example}


\subsection{The Eckmann-Hilton argument}

Structures with binary operations $\rightarrow$ and $\uparrow$
satisfying the interchange law arose during the late 1950s and early 1960s,
in universal algebra, algebraic topology, and double category theory.
The operations are usually associative, but even without this assumption,
Eckmann \& Hilton \cite{EH1962} showed that if we allow them to possess unit elements
$e_\rightarrow$ and $e_\uparrow$ then the interchange law forces the units to be equal:
\begin{align*}
e_\rightarrow
=
e_\rightarrow \rightarrow e_\rightarrow
&=
( e_\rightarrow \uparrow e_\uparrow ) \rightarrow ( e_\uparrow \uparrow e_\rightarrow )
\\
&=
( e_\rightarrow \rightarrow e_\uparrow ) \uparrow ( e_\uparrow \rightarrow e_\rightarrow )
=
e_\uparrow \uparrow e_\uparrow
=
e_\uparrow.
\end{align*}
From this, it further follows that each operation is the opposite of the other,
and that the two operations coincide:
if we write $e = e_\rightarrow = e_\uparrow$ then
\begin{align*}
&
x \rightarrow y
\equiv
( e \uparrow x ) \rightarrow ( y \uparrow e )
\equiv
( e \rightarrow y ) \uparrow ( x \rightarrow e )
\equiv
y \uparrow x,
\\
&
y \uparrow x
\equiv
( y \rightarrow e ) \uparrow ( e \rightarrow x )
\equiv
( y \uparrow e ) \rightarrow ( e \uparrow x )
\equiv
y \rightarrow x.
\end{align*}
Thus there remains \emph{one commutative operation},
which is in fact also \emph{associative}:
\[
( a b ) c \equiv ( a b ) ( e c ) \equiv ( a e ) ( b c ) \equiv a ( b c ).
\]
Hence we assume that \emph{at most one} of the operations possesses a unit element.

\begin{theorem}
Let $S$ be an $d$-tuple interchange magma with operations $\star_1, \dots, \star_d$ for $d \ge 2$.
If these operations have unit elements, then the units are equal, the operations coincide,
and the remaining operation is commutative and associative.
\end{theorem}

\begin{proof}
See the paper of Eckmann \& Hilton \cite{EH1962}, especially
Theorem 3.33 (page 236), the definition of $\mathbf{H}$-structure (page 241), and Theorem 4.17 (page 244).
\end{proof}


\subsection{Double categories}

Extension of the notion of semigroup to sets with $d \ge 2$ operations received
a strong impetus in the 1960s from different approaches to two-dimensional category theory:
see Ehresmann \cite{Ehresmann1963} for double categories,
B\'enabou \cite{Benabou1967} for bicategories,
Kelly \& Street \cite{KS1974} for 2-categories.
The survey by Street \cite{Street1996} also covers higher-dimensional categories;
pasting diagrams \cite{Johnson1989,Power1991} and parity complexes \cite{Street1991}
arose as extensions of the interchange law to higher dimensions.

For our purposes, the most relevant concept is that of double category;
we mention in particular the work of Dawson \& Par\'e \cite{DP1993}.
The most natural example is the double category $\mathbf{Cat}$ which has small categories as objects,
functors as 1-morphisms, and natural transformations as 2-morphisms;
it has two associative operations,
horizontal composition of functors and vertical composition of natural transformations,
which satisfy the interchange law.

\begin{definition}
A \emph{double category} $\mathbf{D}$ is an ordered pair of categories $( \mathbf{D}_0, \mathbf{D}_1 )$,
together with functors
$e \colon \mathbf{D}_0 \to \mathbf{D}_1$ and $s, t \colon \mathbf{D}_1 \to \mathbf{D}_0$.
In $\mathbf{D}_0$ we denote objects by capital Latin letters $A$, $A'$, \dots
(the 0-cells of $\mathbf{D}$)
and morphisms by arrows labelled with lower-case italic letters $u$, $v$, \dots
(the vertical 1-cells of $\mathbf{D}$).
In $\mathbf{D}_1$ we denote objects by arrows labelled with
lower-case italic letters $h$, $k$, \dots
(the horizontal 1-cells of $\mathbf{D}$),
and the morphisms by lower-case Greek letters $\alpha$, $\beta$, \dots
(the 2-cells of $\mathbf{D}$).
If $A$ is an object in $\mathbf{D}_0$ then $e(A)$ is the (horizontal) identity arrow on $A$.
(Recall by the Eckmann-Hilton argument that identity arrows may exist in only one direction.)
The functors $s$ and $t$ are the source and target:
if $h$ is a horizontal arrow in $\mathbf{D}_1$ then $s(h)$ and $t(h)$ are its domain and codomain,
objects in $\mathbf{D}_0$.
These three functors are related by the equation $s(e(A)) = t(e(A)) = A$ for every $A$.
For 2-cells $\alpha$, $\beta$ horizontal composition $\alpha \boxmid \beta$ and vertical composition
$\alpha \boxminus \beta$ are defined by the following diagrams and satisfy the interchange law:
\begin{equation*}
\adjustbox{valign=m}{%
\begin{xy}
(  0,  0 )*+{A} = "a";
( 12,  0 )*+{C} = "c";
( 24,  0 )*+{E} = "e";
(  0, 12 )*+{B} = "b";
( 12, 12 )*+{D} = "d";
( 24, 12 )*+{F} = "f";
(  6,   1 ) = "ac";
(  6,  11 ) = "bd";
( 18,   1 ) = "ce";
( 18,  11 ) = "df";
{ \ar@{->}^{u} "a"; "b" };
{ \ar@{->}_{v} "c"; "d" };
{ \ar@{->}_{w} "e"; "f" };
{ \ar@{->}_{h} "a"; "c" };
{ \ar@{->}_{k} "c"; "e" };
{ \ar@{->}^{\ell} "b"; "d" };
{ \ar@{->}^{m} "d"; "f" };
{ \ar@{=>}_{\alpha} "ac"; "bd" };
{ \ar@{=>}_{\beta}  "ce"; "df" };
\end{xy}
}
\!\!\cong
\adjustbox{valign=m}{%
\begin{xy}
(  0,  0 )*+{A} = "a";
( 18,  0 )*+{E} = "e";
(  0, 12 )*+{B} = "b";
( 18, 12 )*+{F} = "f";
(  9,   1 ) = "ae";
(  9,  11 ) = "bf";
{ \ar@{->}^{u} "a"; "b" };
{ \ar@{->}_{w} "e"; "f" };
{ \ar@{->}_{k \circ h} "a"; "e" };
{ \ar@{->}^{m \circ \ell} "b"; "f" };
{ \ar@{=>}_{\alpha \boxmid \beta} "ae"; "bf" };
\end{xy}
}
\quad
\adjustbox{valign=m}{%
\begin{xy}
(  0,  0 )*+{A} = "a";
(  0, 12 )*+{B} = "b";
(  0, 24 )*+{C} = "c";
( 12,  0 )*+{D} = "d";
( 12, 12 )*+{E} = "e";
( 12, 24 )*+{F} = "f";
(  6,   1 ) = "ad";
(  6,  11 ) = "be-";
(  6,  13 ) = "be+";
(  6,  23 ) = "cf";
{ \ar@{->}^{u} "a"; "b" };
{ \ar@{->}^{v} "b"; "c" };
{ \ar@{->}_{w} "d"; "e" };
{ \ar@{->}_{x} "e"; "f" };
{ \ar@{->}_{h} "a"; "d" };
{ \ar@{->}_{} "b"; "e" };
{ \ar@{->}^{\ell} "c"; "f" };
{ \ar@{=>}_{\alpha} "ad"; "be-" };
{ \ar@{=>}_{\beta}  "be+"; "cf" };
\end{xy}
}
\!\!\!\!\cong
\adjustbox{valign=m}{%
\begin{xy}
(  0,  0 )*+{A} = "a";
(  0, 12 )*+{C} = "c";
( 18,  0 )*+{D} = "d";
( 18, 12 )*+{F} = "f";
(  9,   1 ) = "h";
(  9,  11 ) = "\ell";
{ \ar@{->}_{h} "a"; "d" };
{ \ar@{->}^{\ell} "c"; "f" };
{ \ar@{->}^{v \circ u} "a"; "c" };
{ \ar@{->}_{x \circ w} "d"; "f" };
{ \ar@{=>}_{\alpha \boxminus \beta} "h"; "\ell" };
\end{xy}
}
\end{equation*}
A monoid may be viewed as a category with one object;
this restriction guarantees that all morphisms are composable.
Similarly, a double interchange semigroup may be viewed as a double category with one object.
If we retain associativity and omit interchange, then we obtain a sesquicategory;
see Stell \cite{Stell1995}.
\end{definition}


\subsection{Endomorphism PROPs}

In the category of vector spaces and linear maps over a field $\mathbb{F}$,
the endomorphism PROP of $V$ is the bigraded direct sum
\[
\mathbf{End}(V) =
\bigoplus_{p, q \ge 0} \mathbf{End}(V)^{p,q} =
\bigoplus_{p, q \ge 0} \mathbf{Lin}( V^{\otimes p}, V^{\otimes q} ),
\]
where $\mathbf{Lin}( V^{\otimes p}, V^{\otimes q} )$ is the vector space of all linear maps
$V^{\otimes p} \longrightarrow V^{\otimes q}$.
On $\mathbf{End}(V)$ there are two natural bilinear operations:
\begin{itemize}[leftmargin=*]
\item
The horizontal product:
for $f\colon V^{\otimes p} \longrightarrow V^{\otimes q}$
and $g\colon V^{\otimes r} \longrightarrow V^{\otimes s}$ we define the operation
$\otimes\colon \mathbf{End}(V)^{p,q} \otimes \mathbf{End}(V)^{r,s} \longrightarrow \mathbf{End}(V)^{p+r,q+s}$
as follows:
\[
f \otimes g \colon
V^{\otimes (p+r)} \cong V^{\otimes p} \otimes V^{\otimes r}
\longrightarrow
V^{\otimes q} \otimes V^{\otimes s} \cong V^{\otimes (q+s)}.
\]
\item
The vertical product:
for $f\colon V^{\otimes p} \longrightarrow V^{\otimes q}$
and $g\colon V^{\otimes q} \longrightarrow V^{\otimes r}$ we define
the operation
$\circ \colon \mathbf{End}(V)^{p,q} \otimes \mathbf{End}(V)^{q,r} \longrightarrow \mathbf{End}(V)^{p,r}$
as follows:
\[
g \circ f \colon V^{\otimes p} \longrightarrow V^{\otimes r}.
\]
\end{itemize}
These two operations satisfy the interchange law.
If $f\colon V \rightarrow V'$ and $g\colon W \rightarrow W'$ then
$f \otimes g \colon V \otimes W \longrightarrow V' \otimes W'$
is defined by interchange between $\otimes$ and function \emph{evaluation}:
$( f \otimes g )( v \otimes w ) \equiv f(v) \otimes g(w)$.
If $f'\colon V' \rightarrow V''$ and $g'\colon W' \rightarrow W''$ then
composition of tensor products of maps is defined by interchange between
$\otimes$ and function \emph{composition}:
$( f' \otimes g' ) \circ ( f \otimes g ) \equiv ( f' \circ f ) \otimes ( g' \circ g )$.


\subsection{Tree sequences and Thompson's group}

We consider the group of symmetries of the set of all dyadic partitions of
the open unit interval $I = (0,1)$.

\begin{definition} \label{defdyadicsubset}
A number $x \in I$ is \emph{dyadic of level $b$} if $x = a 2^{-b}$ for positive integers $a, b$
where $a$ is \emph{odd} and $1 \le a \le 2^b{-}1$.
A dyadic subset $C \subset I$ is a \emph{tree sequence} (or \emph{dyadic partition})
if $C$ is obtained from $I$ by a sequence of (exact) bisections of open subintervals.
(Thus $C$ is the image of an unlabelled plane rooted complete binary tree under
the one-dimensional geometric realization map.)
For every $a 2^{-b} \in C$, exactly one of $a{-}1$, $a{+}1$ is twice an odd number,
say $2a'$, and the other is divisible by 4.
Then a dyadic subset is a tree sequence if and only if
$x = a 2^{-b} \in C$ implies $p(x) = a' 2^{-b+1} \in C$;
that is, every $x \in C$ has a tree parent in $C$.
\end{definition}

\begin{definition}
Let $f$ be a homeomorphism of $[0,1]$ which fixes the endpoints and is piecewise linear.
Assume that the subset of $(0,1)$ at which $f$ is not differentiable is a tree sequence,
and that at all other interior points $f'(x)$ is a power of 2.
The set of all such $f$ is a group under function composition,
called \emph{Thompson's group} $F$.
For further information, see Cannon et al.~\cite{CFP1996}.
\end{definition}

Let $A = \{ a_1, \dots, a_n \}$ and $B = \{ b_1, \dots, b_n \}$ be (strictly increasing)
tree sequences of size $n$ partitioning $(0,1)$ into $n{+}1$ subintervals.
We have $f(A) = B$ where $f \in F$ is linear on each subinterval and 
satisfies $f(a_i) = b_i$ for $1 \le i \le n$.
Thus $F$ describes transformations from one rooted binary tree to another.
Plane rooted complete binary trees with $n$ internal nodes are in bijection with association types
for nonassociative products of $n{+}1$ factors.
Hence, we may also regard $F$ as consisting of transformations from one association type to another;
in this case, we call $f \in F$ a \emph{reassociation} of the parentheses.
We display the bijection between tree sequences and association types for arities $\le 5$ in Figure \ref{treesequences}.

\begin{figure}[ht]
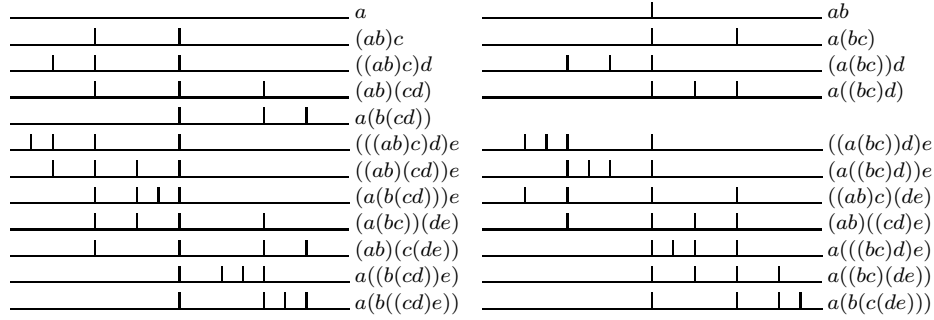

{\footnotesize
\[
\begin{array}{l@{\;}l@{\quad}l@{\;}l@{\,}}
\line(1,0){128} & a
&
\line(1,0){64}\line(0,1){6}\line(1,0){64} & ab
\\
\line(1,0){32}\line(0,1){6}\line(1,0){32}\line(0,1){6}\line(1,0){64} & (ab)c
&
\line(1,0){64}\line(0,1){6}\line(1,0){32}\line(0,1){6}\line(1,0){32} & a(bc)
\\
\line(1,0){16}\line(0,1){6}\line(1,0){16}\line(0,1){6}\line(1,0){32}\line(0,1){6}\line(1,0){64} & ((ab)c)d
&
\line(1,0){32}\line(0,1){6}\line(1,0){16}\line(0,1){6}\line(1,0){16}\line(0,1){6}\line(1,0){64} & (a(bc))d
\\
\line(1,0){32}\line(0,1){6}\line(1,0){32}\line(0,1){6}\line(1,0){32}\line(0,1){6}\line(1,0){32} & (ab)(cd)
&
\line(1,0){64}\line(0,1){6}\line(1,0){16}\line(0,1){6}\line(1,0){16}\line(0,1){6}\line(1,0){32} & a((bc)d)
\\
\line(1,0){64}\line(0,1){6}\line(1,0){32}\line(0,1){6}\line(1,0){16}\line(0,1){6}\line(1,0){16} & a(b(cd))
\\
\line(1,0){8}\line(0,1){6}\line(1,0){8}\line(0,1){6}\line(1,0){16}\line(0,1){6}\line(1,0){32}\line(0,1){6}\line(1,0){64}
 & (((ab)c)d)e
&
\line(1,0){16}\line(0,1){6}\line(1,0){8}\line(0,1){6}\line(1,0){8}\line(0,1){6}\line(1,0){32}\line(0,1){6}\line(1,0){64}
 & ((a(bc))d)e
\\
\line(1,0){16}\line(0,1){6}\line(1,0){16}\line(0,1){6}\line(1,0){16}\line(0,1){6}\line(1,0){16}\line(0,1){6}\line(1,0){64}
 & ((ab)(cd))e
&
\line(1,0){32}\line(0,1){6}\line(1,0){8}\line(0,1){6}\line(1,0){8}\line(0,1){6}\line(1,0){16}\line(0,1){6}\line(1,0){64}
 & (a((bc)d))e
\\
\line(1,0){32}\line(0,1){6}\line(1,0){16}\line(0,1){6}\line(1,0){8}\line(0,1){6}\line(1,0){8}\line(0,1){6}\line(1,0){64}
 & (a(b(cd)))e
&
\line(1,0){16}\line(0,1){6}\line(1,0){16}\line(0,1){6}\line(1,0){32}\line(0,1){6}\line(1,0){32}\line(0,1){6}\line(1,0){32}
 & ((ab)c)(de)
\\
\line(1,0){32}\line(0,1){6}\line(1,0){16}\line(0,1){6}\line(1,0){16}\line(0,1){6}\line(1,0){32}\line(0,1){6}\line(1,0){32}
 & (a(bc))(de)
&
\line(1,0){32}\line(0,1){6}\line(1,0){32}\line(0,1){6}\line(1,0){16}\line(0,1){6}\line(1,0){16}\line(0,1){6}\line(1,0){32}
 & (ab)((cd)e)
\\
\line(1,0){32}\line(0,1){6}\line(1,0){32}\line(0,1){6}\line(1,0){32}\line(0,1){6}\line(1,0){16}\line(0,1){6}\line(1,0){16}
 & (ab)(c(de))
&
\line(1,0){64}\line(0,1){6}\line(1,0){8}\line(0,1){6}\line(1,0){8}\line(0,1){6}\line(1,0){16}\line(0,1){6}\line(1,0){32}
 & a(((bc)d)e)
\\
\line(1,0){64}\line(0,1){6}\line(1,0){16}\line(0,1){6}\line(1,0){8}\line(0,1){6}\line(1,0){8}\line(0,1){6}\line(1,0){32}
 & a((b(cd))e)
&
\line(1,0){64}\line(0,1){6}\line(1,0){16}\line(0,1){6}\line(1,0){16}\line(0,1){6}\line(1,0){16}\line(0,1){6}\line(1,0){16}
 & a((bc)(de))
\\
\line(1,0){64}\line(0,1){6}\line(1,0){32}\line(0,1){6}\line(1,0){8}\line(0,1){6}\line(1,0){8}\line(0,1){6}\line(1,0){16}
 & a(b((cd)e))
&
\line(1,0){64}\line(0,1){6}\line(1,0){32}\line(0,1){6}\line(1,0){16}\line(0,1){6}\line(1,0){8}\line(0,1){6}\line(1,0){8}
 & a(b(c(de)))
\end{array}
\]}
\vspace{-5mm}
\caption{Tree sequences and association types}
\label{treesequences}
\end{figure}


\section{Preliminary results on commutativity relations}


\subsection{Lemmas on associativity and interchange}

For $n \ge 1$, the tree monomial basis of $\mathbf{Free}(n)$ is the set $\mathbb{B}_n$ of all complete
rooted binary plane trees with $n$ leaves, with internal nodes labelled $\wedgehor$ or $\wedgever$,
and leaves labelled by a permutation of $x_1, \dots, x_n$ (Definition \ref{freeoperad}).
For $\mathbf{Assoc}$, either we use equivalence classes (under double associativity)
of binary trees as basis monomials, or (more conveniently) we use the basis
$\mathbb{NB} = \{ x_1 \} \sqcup \mathbb{T}_\wedgehor \sqcup \mathbb{T}_\wedgever$
of labelled rooted (not necessarily binary) plane trees with \emph{alternating} labels
$\wedgehor$ and $\wedgever$ on internal nodes (Definitions \ref{assocboperad}, \ref{assocnboperad}).

\begin{lemma}
\label{assoclemma}
A basis for $\mathbf{Assoc}(n)$ is the set $\mathbb{NB}_n$ of all trees in $\mathbb{NB}$ with $n$ leaves.
\end{lemma}

\begin{proof}
We give an algorithm for converting a tree $T \in \mathbb{B}_n$ into a tree $\alpha(T) \in \mathbb{NB}_n$.
We omit the (trivial but tedious) details of the proof that $\alpha$ is surjective,
and that for any tree $U \in \mathbb{NB}_n$, the inverse image $\alpha^{-1}(U) \subseteq \mathbb{B}_n$
consists of a single equivalence class for the congruence on $\mathbb{B}_n$ defined by
the consequences of the associativity relations $\HA$, $\VA$ of equation \eqref{associativelaws}.
We define $\alpha$ by the following diagrams, which indicate that for every $T \in \mathbb{B}_n$,
and \emph{every} internal node labelled $\wedgehor$, the subtree of $T$ with that node as root
is rewritten as indicated, obtaining a tree $\alpha(T) \in \mathbb{NB}_n$:
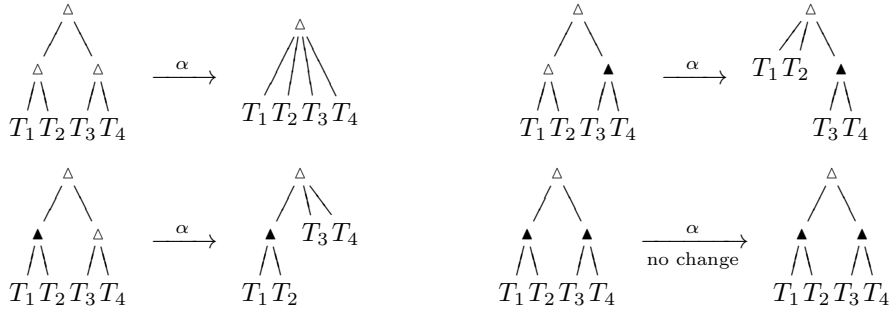
\[
\begin{array}{c@{\qquad\qquad}c}
\adjustbox{valign=m}{
\begin{xy}
(  6, 16 )*+{\wedgehor} = "root";
(  2,  8 )*+{\wedgehor} = "l";
( 10,  8 )*+{\wedgehor} = "r";
(  0,  0 )*+{T_1} = "t1";
(  4,  0 )*+{T_2} = "t2";
(  8,  0 )*+{T_3} = "t3";
( 12,  0 )*+{T_4} = "t4";
{ \ar@{-} "root"; "l" };
{ \ar@{-} "root"; "r" };
{ \ar@{-} "l"; "t1" };
{ \ar@{-} "l"; "t2" };
{ \ar@{-} "r"; "t3" };
{ \ar@{-} "r"; "t4" };
\end{xy}
}
\xrightarrow{\;\;\;\alpha\;\;\;}
\adjustbox{valign=m}{
\begin{xy}
(  6, 12 )*+{\wedgehor} = "root";
(  0,  0 )*+{T_1} = "t1";
(  4,  0 )*+{T_2} = "t2";
(  8,  0 )*+{T_3} = "t3";
( 12,  0 )*+{T_4} = "t4";
{ \ar@{-} "root"; "t1" };
{ \ar@{-} "root"; "t2" };
{ \ar@{-} "root"; "t3" };
{ \ar@{-} "root"; "t4" };
\end{xy}
}
&
\adjustbox{valign=m}{
\begin{xy}
(  6, 16 )*+{\wedgehor} = "root";
(  2,  8 )*+{\wedgehor} = "l";
( 10,  8 )*+{\wedgever} = "r";
(  0,  0 )*+{T_1} = "t1";
(  4,  0 )*+{T_2} = "t2";
(  8,  0 )*+{T_3} = "t3";
( 12,  0 )*+{T_4} = "t4";
{ \ar@{-} "root"; "l" };
{ \ar@{-} "root"; "r" };
{ \ar@{-} "l"; "t1" };
{ \ar@{-} "l"; "t2" };
{ \ar@{-} "r"; "t3" };
{ \ar@{-} "r"; "t4" };
\end{xy}
}
\xrightarrow{\;\;\;\alpha\;\;\;}
\adjustbox{valign=m}{
\begin{xy}
(  6, 16 )*+{\wedgehor} = "root";
( 10,  8 )*+{\wedgever} = "r";
(  0,  8 )*+{T_1} = "t1";
(  4,  8 )*+{T_2} = "t2";
(  8,  0 )*+{T_3} = "t3";
( 12,  0 )*+{T_4} = "t4";
{ \ar@{-} "root"; "t1" };
{ \ar@{-} "root"; "t2" };
{ \ar@{-} "root"; "r" };
{ \ar@{-} "r"; "t3" };
{ \ar@{-} "r"; "t4" };
\end{xy}
}
\\[9mm]
\adjustbox{valign=m}{
\begin{xy}
(  6, 16 )*+{\wedgehor} = "root";
(  2,  8 )*+{\wedgever} = "l";
( 10,  8 )*+{\wedgehor} = "r";
(  0,  0 )*+{T_1} = "t1";
(  4,  0 )*+{T_2} = "t2";
(  8,  0 )*+{T_3} = "t3";
( 12,  0 )*+{T_4} = "t4";
{ \ar@{-} "root"; "l" };
{ \ar@{-} "root"; "r" };
{ \ar@{-} "l"; "t1" };
{ \ar@{-} "l"; "t2" };
{ \ar@{-} "r"; "t3" };
{ \ar@{-} "r"; "t4" };
\end{xy}
}
\xrightarrow{\;\;\;\alpha\;\;\;}
\adjustbox{valign=m}{
\begin{xy}
(  6, 16 )*+{\wedgehor} = "root";
(  2,  8 )*+{\wedgever} = "l";
(  0,  0 )*+{T_1} = "t1";
(  4,  0 )*+{T_2} = "t2";
(  8,  8 )*+{T_3} = "t3";
( 12,  8 )*+{T_4} = "t4";
{ \ar@{-} "root"; "l" };
{ \ar@{-} "root"; "t3" };
{ \ar@{-} "root"; "t4" };
{ \ar@{-} "l"; "t1" };
{ \ar@{-} "l"; "t2" };
\end{xy}
}
&
\adjustbox{valign=m}{
\begin{xy}
(  6, 16 )*+{\wedgehor} = "root";
(  2,  8 )*+{\wedgever} = "l";
( 10,  8 )*+{\wedgever} = "r";
(  0,  0 )*+{T_1} = "t1";
(  4,  0 )*+{T_2} = "t2";
(  8,  0 )*+{T_3} = "t3";
( 12,  0 )*+{T_4} = "t4";
{ \ar@{-} "root"; "l" };
{ \ar@{-} "root"; "r" };
{ \ar@{-} "l"; "t1" };
{ \ar@{-} "l"; "t2" };
{ \ar@{-} "r"; "t3" };
{ \ar@{-} "r"; "t4" };
\end{xy}
}
\xrightarrow[\text{no change}]{\;\alpha\;}
\adjustbox{valign=m}{
\begin{xy}
(  6, 16 )*+{\wedgehor} = "root";
(  2,  8 )*+{\wedgever} = "l";
( 10,  8 )*+{\wedgever} = "r";
(  0,  0 )*+{T_1} = "t1";
(  4,  0 )*+{T_2} = "t2";
(  8,  0 )*+{T_3} = "t3";
( 12,  0 )*+{T_4} = "t4";
{ \ar@{-} "root"; "l" };
{ \ar@{-} "root"; "r" };
{ \ar@{-} "l"; "t1" };
{ \ar@{-} "l"; "t2" };
{ \ar@{-} "r"; "t3" };
{ \ar@{-} "r"; "t4" };
\end{xy}
}
\end{array}
\]
Switching $\wedgehor$ and $\wedgever$ throughout defines $\alpha$ for subtrees with roots
labelled $\wedgever$.
\end{proof}

Gu et al.~\cite{GLM2008} study various classes of binary trees whose internal nodes are labelled
white or black; however, none of their results coincides with our situation.
The Knuth rotation correspondence is similar but not identical to our bijection $\alpha$
between binary and nonbinary trees;
see Ebrahimi-Fard \& Manchon \cite[\S2]{EFM2014}.

For $n \ge 1$, consider the graph $G_n$ whose vertex set is $\mathbb{B}_n$;
the size of this set is the large Schr\"oder numbers (OEIS A006318):
1, 2, 6, 22, 90, 394, 1806, \dots.
In $G_n$ there is an edge joining tree monomials $v, w$ if and only if
$w$ may be obtained from $v$ by one application of the interchange law.
The number of \emph{isolated} vertices in $G_n$ \cite{BM2016} is the sequence (OEIS A078482)
1, 2, 6, 20, 70, 254, 948, \dots.
This is also the number of planar guillotine partitions of $I^2$ which avoid
a certain nondyadic block partition with four parts 
(equivalently, there is no way to apply the interchange law);
see Asinowski et al.~\cite[\S6.2, Remark 1]{ABBMP2013}.

\begin{notation}
For monomials $m_1, m_2 \in \mathbf{Free}(n)$ with $n \ge 4$, we write $m_1 \equiv m_2$
if and only if $m_1$ and $m_2$ can be obtained from the two sides of the interchange law
\eqref{intlaw} by the same sequence of partial compositions.
We write $m_1 \sim m_2$ if and only if $\Gamma( m_1 ) = \Gamma( m_2 )$,
where $\Gamma$ is the geometric realization map (Definition \ref{defgeomap}).
\end{notation}

\begin{lemma}
\label{nasrinslemma}
The equivalence relations $\sim$ and $\equiv$ coincide.
That is, $\sim$ is generated by the consequences in arity $n$ of the interchange law \eqref{intlaw}.
\end{lemma}

\begin{proof}
For $n = 1,2,3$, the map $\Gamma$ is injective, so there is nothing to prove.
Now suppose that $n \ge 4$ and that $m_1, m_2 \in \mathbf{Free}(n)$ satisfy $m_1 \sim m_2$;
thus for some dyadic block partition $P \in \mathbf{DBP}(n)$ we have
$m_1, m_2 \in \Gamma^{-1}(P)$.

For $n = 4$, the dihedral group of symmetries of the square acts on the basis of 40 tree monomials;
the generators are replacing $\wedgehor$ (resp.~$\wedgever$) by the opposite operation and
transposing the operations.
In the following argument, we omit the permutations of the indeterminates,
but the reasoning remains valid for a symmetric operad.
There are nine orbits, of sizes two (twice), four (five times), eight (twice).
For each orbit, we choose an orbit representative and display
its image under $\Gamma$ in Figure \ref{rectangularpartitions}.
The dihedral group also acts in the obvious way on these nine dyadic block partitions.
In every case except the first, the size of the orbit generated by the block partition
equals the size of the orbit generated by the tree monomial.
The first block partition $\boxplus$ is fixed by all 8 symmetrices of the square,
and the two monomials in $\Gamma^{-1}(\boxplus)$ are the two terms of the interchange law
\eqref{intlaw}.
This is the only failure of injectivity in arity 4.

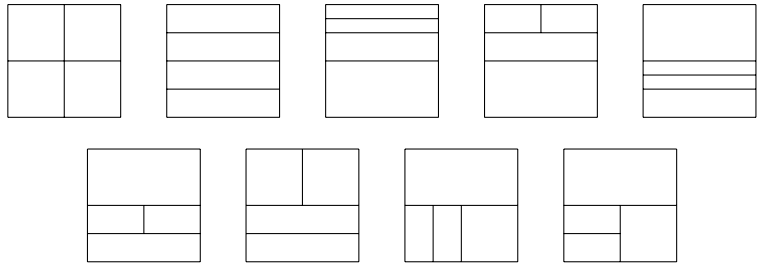
\begin{figure}[ht]
\footnotesize
\[
\begin{array}{c}
\begin{tikzpicture}
\draw
(0,0) -- (1.5,0)
(0,1.5) -- (1.5,1.5)
(0,0.75) -- (1.5,0.75)
(0,0) -- (0,1.5)
(0.75,0) -- (0.75,1.5)
(1.5,0) -- (1.5,1.5);
\end{tikzpicture}
\qquad
\begin{tikzpicture}
\draw
(0,0) -- (1.5,0)
(0,1.5) -- (1.5,1.5)
(0,1.125) -- (1.5,1.125)
(0,0.75) -- (1.5,0.75)
(0,0.375) -- (1.5,0.375)
(0,0) -- (0,1.5)
(1.5,0) -- (1.5,1.5);
\end{tikzpicture}
\qquad
\begin{tikzpicture}
\draw
(0,0) -- (1.5,0)
(0,1.5) -- (1.5,1.5)
(0,1.3125) -- (1.5,1.3125)
(0,1.125) -- (1.5,1.125)
(0,0.75) -- (1.5,0.75)
(0,0) -- (0,1.5)
(1.5,0) -- (1.5,1.5);
\end{tikzpicture}
\qquad
\begin{tikzpicture}
\draw
(0,0) -- (1.5,0)
(0,1.5) -- (1.5,1.5)
(0,1.125) -- (1.5,1.125)
(0,0.75) -- (1.5,0.75)
(0,0) -- (0,1.5)
(0.75,1.125) -- (0.75,1.5)
(1.5,0) -- (1.5,1.5);
\end{tikzpicture}
\qquad
\begin{tikzpicture}
\draw
(0,0) -- (1.5,0)
(0,0.375) -- (1.5,0.375)
(0,0.5625) -- (1.5,0.5625)
(0,0.75) -- (1.5,0.75)
(0,1.5) -- (1.5,1.5)
(0,0) -- (0,1.5)
(1.5,0) -- (1.5,1.5);
\end{tikzpicture}
\\[3mm]
\begin{tikzpicture}
\draw
(0,0) -- (1.5,0)
(0,0.375) -- (1.5,0.375)
(0,0.75) -- (1.5,0.75)
(0,1.5) -- (1.5,1.5)
(0,0) -- (0,1.5)
(0.75,0.375) -- (0.75,0.75)
(1.5,0) -- (1.5,1.5);
\end{tikzpicture}
\qquad
\begin{tikzpicture}
\draw
(0,0) -- (1.5,0)
(0,0.375) -- (1.5,0.375)
(0,0.75) -- (1.5,0.75)
(0,1.5) -- (1.5,1.5)
(0,0) -- (0,1.5)
(0.75,0.75) -- (0.75,1.5)
(1.5,0) -- (1.5,1.5);
\end{tikzpicture}
\qquad
\begin{tikzpicture}
\draw
(0,0) -- (1.5,0)
(0,1.5) -- (1.5,1.5)
(0,0.75) -- (1.5,0.75)
(0,0) -- (0,1.5)
(0.75,0) -- (0.75,0.75)
(0.375,0) -- (0.375,0.75)
(1.5,0) -- (1.5,1.5);
\end{tikzpicture}
\qquad
\begin{tikzpicture}
\draw
(0,0) -- (1.5,0)
(0,1.5) -- (1.5,1.5)
(0,0.75) -- (1.5,0.75)
(0,0.375) -- (0.75,0.375)
(0,0) -- (0,1.5)
(0.75,0) -- (0.75,0.75)
(1.5,0) -- (1.5,1.5);
\end{tikzpicture}
\end{array}
\]
\vspace{-5mm}
\caption{Orbit representatives for dihedral group in arity 4}
\label{rectangularpartitions}
\end{figure}

Assume that $n \ge 5$ and that $\sim$ and $\equiv$ coincide on $\mathbf{Free}(k)$ for $k < n$.
Clearly any monomial $m \in \mathbf{Free}(k)$ has the form $m = m_1 \ast m_2$ for some
$m_1 \in \mathbf{Free}(k_1)$, $m_2 \in \mathbf{Free}(k_2)$ where
$k_1, k_2 < n$, $k_1 + k_2 = n$, and $\ast \in \{ \wedgehor, \wedgever \}$.
Consider a dyadic block partition $P \in \mathbf{DBP}(n)$.
There are three cases:

\emph{Case 1}.
Assume $P$ contains the horizontal bisection of $I^2$, but the two resulting parts
do not both have vertical bisections.
Let $x, y \in \mathbf{Free}(n)$ be tree monomials in $\Gamma^{-1}(P)$, so $x \sim y$.
By assumption, we have $x = x_1 \wedgever x_2$ where $x_i = x'_i \wedgever x''_i$
for at most one $i \in \{1,2\}$; and the same for $y$.
Since $\Gamma$ is an operad morphism, it follows from
$\Gamma( x_1 \wedgever x_2 ) = \Gamma( y_1 \wedgever y_2 )$ that
$\Gamma( x_1 ) \uparrow \Gamma( x_2 ) = \Gamma( y_1 ) \uparrow \Gamma( y_2 )$.
It is geometrically clear that $\Gamma( x_i ) = \Gamma( y_i )$ for $i \in \{1,2\}$,
and this implies that $x_i$ and $y_i$ have the same arity $k_i$.
Hence $x_i \sim y_i$, and by induction $x_i \equiv y_i$.
Therefore $x \equiv y$.

\emph{Case 2}.
Assume $P$ contains the vertical bisection of $I^2$, but the two resulting parts
do not both have horizontal bisections.
The argument is the same as Case 1 with $\wedgehor$ and $\wedgever$ transposed;
this leaves the interchange law \eqref{intlaw} unchanged.

\emph{Case 3}.
Assume $P$ contains both horizontal and vertical bisections of $I^2$.
In addition to the possibilities in Cases 1 and 2,
there are two different factorizations for each monomial $x, y \in \Gamma^{-1}(P)$
into products of four factors.
Using both algebraic and geometric notation, we have:
\begin{align*}
&
x =
x_1 \wedgever x_2 =
( z_1 \wedgehor z_2 ) \wedgever ( z_3 \wedgehor z_4 )
\stackrel{\boxplus}{=}
( z_1 \wedgever z_3 ) \wedgehor ( z_2 \wedgever z_4 ) =
y_1 \wedgehor y_2 =
y,
\\
&
\Gamma(x)
=
\adjustbox{valign=m}
{\begin{xy}
( 15,  5 )*+{\Gamma(x_1)};
( 15, 15 )*+{\Gamma(x_2)};
( 10,  0 ) = "1";
( 10, 10 ) = "2";
( 10, 20 ) = "3";
( 20,  0 ) = "4";
( 20, 10 ) = "5";
( 20, 20 ) = "6";
{ \ar@{-} "1"; "3" };
{ \ar@{-} "4"; "6" };
{ \ar@{-} "1"; "4" };
{ \ar@{-} "2"; "5" };
{ \ar@{-} "3"; "6" };
\end{xy}}
=
\adjustbox{valign=m}
{\begin{xy}
(  5,  5 )*+{\Gamma(z_1)};
(  5, 15 )*+{\Gamma(z_3)};
( 15,  5 )*+{\Gamma(z_2)};
( 15, 15 )*+{\Gamma(z_4)};
(  0,  0 ) = "1";
(  0, 10 ) = "2";
(  0, 20 ) = "3";
( 10,  0 ) = "4";
( 10, 10 ) = "5";
( 10, 20 ) = "6";
( 20,  0 ) = "7";
( 20, 10 ) = "8";
( 20, 20 ) = "9";
{ \ar@{-} "1"; "3" };
{ \ar@{-} "4"; "6" };
{ \ar@{-} "7"; "9" };
{ \ar@{-} "1"; "7" };
{ \ar@{-} "2"; "8" };
{ \ar@{-} "3"; "9" };
\end{xy}}
=
\adjustbox{valign=m}
{\begin{xy}
(  5,  5 )*+{\Gamma(y_1)};
( 15,  5 )*+{\Gamma(y_2)};
(  0,  0 ) = "1";
(  0, 10 ) = "2";
( 10,  0 ) = "3";
( 10, 10 ) = "4";
( 20,  0 ) = "5";
( 20, 10 ) = "6";
{ \ar@{-} "1"; "2" };
{ \ar@{-} "3"; "4" };
{ \ar@{-} "5"; "6" };
{ \ar@{-} "1"; "5" };
{ \ar@{-} "2"; "6" };
\end{xy}}
=
\Gamma(y).
\end{align*}
If $x \sim y$ then either
(i)
the claim follows from the equivalence of factors in lower arity as in Cases 1 and 2,
or
(ii)
the claim follows from an application of the interchange law in arity $n$ as indicated in
the last two equations.
\end{proof}

For a generalization of Lemma \ref{nasrinslemma} to $d \ge 2$ nonassociative operations, see \cite{BD2017}.


\subsection{Cuts and slices}

Recall the notions of empty blocks and subrectangles in a block partition from Definitions
\ref{bpoperad} and \ref{defsubrectangle}.

\begin{definition}
Let $P$ be a block partition of $I^2$ and $R$ a subrectangle of $P$.
By a \emph{main cut} in $R$ we mean a horizontal or vertical bisection of $R$.
Every subrectangle has at most two main cuts; the empty block is the only block partition with no main cut.
Suppose that a main cut partitions $R$ into subrectangles $R_1$ and $R_2$.
If either $R_1$ or $R_2$ has a main cut parallel to the main cut of $R$,
this is called a \emph{primary cut} in $R$.
This definition extends as follows:
if the subrectangle $S$ of $R$ is one of the subrectangles obtained from a sequence of cuts
all of which are parallel to a main cut of $R$ then a main cut of $S$ is a primary cut of $R$.
In a given direction, we include the main cut of $R$ as a primary cut.
Let $C_1, \dots, C_\ell$ be all the primary cuts of $R$ parallel to a given main cut $C_i$ of $R$
($1 \le i \le \ell$) in their natural order (bottom to top, or left to right)
so that there is no primary cut between $C_j$ and $C_{j+1}$ for $1 \le j \le \ell{-}1$.
Define the artificial ``cuts'' $C_0$ and $C_{\ell+1}$ to be the bottom and top
(or left and right) sides of $R$.
We write $S_j$ for the $j$-th \emph{slice} of $R$ parallel to the given main cut; that is,
the subrectangle between $C_{j-1}$ and $C_j$ for $1 \le j \le \ell{+}1$.
\end{definition}

\begin{definition}
Let $m$ be a monomial of arity $n$ in the operad $\mathbf{Free}$.
We say that $m$ \emph{admits a commutativity relation} if for some transposition $(ij) \in S_n$ ($i < j$),
the following relation holds for the corresponding cosets in $\mathbf{DIA}$:
\[
m(x_1,\dots,x_i,\dots,x_j,\dots,x_n) \equiv m(x_1,\dots,x_j,\dots,x_i,\dots,x_n).
\]
\end{definition}

We emphasize (referring to the commutative diagram of Figure \ref{bigpicture}) that the proof of
a commutativity property for the monomial $m$ consists of a sequence of applications of associativity and
the interchange law starting from $m$ and ending with the same pattern of parentheses and operations 
but with a different permutation.

\begin{proposition}
\label{twomaincuts}
Let $m$ be a tree monomial in $\mathbf{Free}$ which admits a commutativity relation.
Assume that this commutativity relation is not the result of operad partial composition with
a commutativity relation of lower arity, either from
(i) a commutativity relation holding in a proper factor of $m$, or
(ii) a commutativity relation holding in a proper quotient of $m$, by which we mean substitution of
the same decomposable factor for the same indecomposable argument in both sides of a commutativity relation
of lower arity.
If $P = \Gamma(m)$ is the corresponding dyadic block partition of $I^2$ then $P$ contains both of
the main cuts (horizontal and vertical); that is, it must be possible to apply the interchange law
as a rewrite rule at the root of the tree monomial $m$.
\end{proposition}

\begin{proof}
Any dyadic block partition $P$ of $I^2$ has at least one main cut, corresponding to the root of the
tree monomial $m$ for which $P = \Gamma(m)$.
Transposing the $x$ and $y$ axes if necessary (this corresponds to switching the horizontal and
vertical operation symbols in the monomial $m$), we may assume that $P$ contains the vertical main cut,
corresponding to the operation $\wedgehor$ in the monomial $m = m_1 \wedgehor m_2$:
\[
P = \begin{array}{|c|c|} \midrule P_1 & P_2 \\ \midrule \end{array}
\]
Let $P_1 = \Gamma(m_1)$ and $P_2 = \Gamma(m_2)$ be the dyadic block partitions of $I^2$ corresponding
to $m_1$ and $m_2$.
If both $P_1$ and $P_2$ have the horizontal main cut, then we are done, since these two cuts combine
to produce the horizontal main cut for $P$:
\[
P = \begin{array}{|c|c|} \midrule P''_1 & P''_2 \\ \midrule P'_1 & P'_2 \\ \midrule \end{array}
\]
Otherwise, at most one of $P_1$ and $P_2$ has the horizontal main cut.
Reflecting in the vertical line $x = \tfrac12$ if necessary (this corresponds to replacing
the horizontal operation $\wedgehor$ by its opposite throughout tree monomial $m$), we may assume that
$P_1$ does \emph{not} have the horizontal main cut.
Then either $P_1$ has the vertical main cut, or $P_1$ has no main cut (so that $P_1$ is an empty block):
\[
P = \begin{array}{|c|c|c|} \midrule P'_1 & P''_1 & P_2 \\ \midrule \end{array}
\qquad \text{or} \qquad
P = \begin{array}{|c|c|} \midrule P_1 & P_2 \\ \midrule \end{array}
\]
In either case, $P_1$ is the union of $k \ge 1$ consecutive vertical slices $S_1, \dots, S_k$
from left to right, where we assume that $k$ is as large as possible so that the slices are as
thin as possible.
It follows that each of these vertical slices either is an empty block or has only one main cut
which is horizontal:
\[
P = \begin{array}{|c|c|c|c|} \midrule S_1 & \cdots & S_k & P_2 \\ \midrule \end{array}
\]
Therefore, in the monomial $m_1$ for which $P_1 = \Gamma( m_1 )$, each of these vertical slices
corresponds either to an indecomposable indeterminate $x_j$ or a decomposable element $t \wedgever u$
whose root operation is the vertical operation (corresponding to the horizontal main cut).
By assumption, $P_1$ does not have the horizontal main cut, and so at least one of the vertical slices
$S_j$ does not have the horizontal main cut; we choose $j$ to be as small as possible, thereby selecting
the leftmost vertical slice without the horizontal main cut:
\[
P = \begin{array}{|c|c|c|c|c|c|} \midrule S_1 & \cdots & x_j & \cdots & S_k & P_2 \\ \midrule \end{array}
\]
By maximality of the choice of $k$, the vertical slice $S_j$ does not have the vertical main cut either.
Thus $S_j$ has no main cut, and hence $S_j$ is the empty block, and so in the monomial $m_1$,
the vertical slice $S_j$ corresponds to an indecomposable indeterminate $x_j$.
Therefore $m$ must have the following form, where $v$ and/or $w$ may be absent
(that is, $v \wedgehor x_j \wedgehor w$ may be $x_j \wedgehor w$ or $v \wedgehor x_j$ or simply $x_j$):
\begin{equation}
\label{vwequation}
m = m_1 \wedgehor m_2 = v \wedgehor x_j \wedgehor w \wedgehor m_2.
\end{equation}
(We may omit parentheses since the operation $\wedgehor$ is associative.)

If both $v$ and $w$ are absent, then $m_1 = x_j$ and so $m = x_j \wedgehor m_2$.
In this case, it is clear that the only way in which the interchange law can be applied as a rewrite rule
to $m$ is within the submonomial $m_2$.
But this implies that any commutativity relation which holds for $m$ is a consequence of
a commutativity relation for $m_2$, contradicting our assumption.

If $x_j$ is not the only argument in $m_1$ then there is at least one factor $v$ or $w$
on the left or right side of $x_j$ in equation \eqref{vwequation}.
We want to be able to apply the interchange law as a rewrite rule in a way which involves all of $m$;
otherwise, any commutativity relation which holds for $m$ must be a consequence of
a commutativity relation for a proper submonomial, contradicting our assumption.
Let us write the monomial \eqref{vwequation} as a tree monomial; it has the form
\begin{equation}
\label{vwtree1}
\adjustbox{valign=m}{
\begin{xy}
( 12, 12 )*+{\wedgehor} = "root";
(  0,  0 )*+{\boxed{v}} = "t1";
(  8,  0 )*+{x_j} = "t2";
( 16,  0 )*+{\boxed{w}} = "t3";
( 24,  0 )*+{\boxed{m_2}} = "t4";
{ \ar@{-} "root"; "t1" };
{ \ar@{-} "root"; "t2" };
{ \ar@{-} "root"; "t3" };
{ \ar@{-} "root"; "t4" };
\end{xy}
}
\end{equation}
We can apply the interchange law to this tree only in one of the following ways:
within $v$, within $w$, within $m_2$, or (if both $w$ and $m_2$ have $\wedgever$ at the root)
using the root of \eqref{vwtree1} with $w$ and $m_2$.
In the last case, we first rewrite \eqref{vwtree1} as follows:
\begin{equation}
\label{vwtree2}
\adjustbox{valign=m}{
\begin{xy}
( 12, 16 )*+{\wedgehor} = "root";
( 20,  8 )*+{\wedgehor} = "r";
(  0,  8 )*+{\boxed{v}} = "t1";
(  8,  8 )*+{x_j} = "t2";
( 16,  0 )*+{\boxed{w}} = "t3";
( 24,  0 )*+{\boxed{m_2}} = "t4";
{ \ar@{-} "root"; "t1" };
{ \ar@{-} "root"; "t2" };
{ \ar@{-} "root"; "r" };
{ \ar@{-} "r"; "t3" };
{ \ar@{-} "r"; "t4" };
\end{xy}
}
\end{equation}
After applying the interchange law, we obtain a tree of the following form:
\begin{equation}
\label{vwtree3}
\adjustbox{valign=m}{
\begin{xy}
( 12, 16 )*+{\wedgehor} = "root";
( 20,  8 )*+{\wedgever} = "r";
(  0,  8 )*+{\boxed{v}} = "t1";
(  8,  8 )*+{x_j} = "t2";
( 16,  0 )*+{\boxed{w'}} = "t3";
( 24,  0 )*+{\boxed{m'_2}} = "t4";
{ \ar@{-} "root"; "t1" };
{ \ar@{-} "root"; "t2" };
{ \ar@{-} "root"; "r" };
{ \ar@{-} "r"; "t3" };
{ \ar@{-} "r"; "t4" };
\end{xy}
}
\end{equation}
Thus, no matter how we apply the interchange law to \eqref{vwtree1}, the fact that $x_j$
is a child of the root $\wedgehor$ remains unchanged.
Hence any commutativity relation for $m$ must be a consequence of a commutativity relation
of lower arity (either as factor or as quotient), contradicting our original assumption.
Therefore both $P_1$ and $P_2$ have the horizontal main cut, and hence $P$ has both main cuts.
\end{proof}

\subsection{Border blocks and interior blocks}

\begin{definition}
Let $m$ be a (tree) monomial in the operad $\mathbf{Free}$ which admits a commutativity relation
transposing the indeterminates $x_i$ and $x_j$ ($i \ne j$).
If $B = \Gamma(m)$ is the labelled block partition of $I^2$ corresponding to $m$ then $B$ is called
a \emph{commutative block partition}.
The two empty blocks corresponding to $x_i$ and $x_j$ are called \emph{commuting empty blocks}.
\end{definition}

\begin{definition}
Let $B$ be a block partition of $I^2$ consisting of the empty blocks $R_1$, $\dots$, $R_k$.
If the closure of $R_i$ has empty intersection with the four sides of the closure $\overline{I^2}$ then
$R_i$ is an \emph{interior block}, otherwise $R_i$ is a \emph{border block}.
\end{definition}

\begin{lemma}
\label{borderlemma}
Suppose that $B_1 = \Gamma(m_1)$ and $B_2 = \Gamma(m_2)$ are two labelled dyadic block partitions of $I^2$
such that $m_1 \equiv m_2$ in every double interchange semigroup; hence this equivalence must be the result
of applying associativity and the interchange law.
(This is more general than a commutativity relation for a dyadic block partition.)
Then any interior (respectively border) block of $B_1$ remains an interior (respectively border) block in $B_2$.
\end{lemma}

\begin{proof}
It is clear from the geometric realizations that neither associativity
nor the interchange law can change an interior block to a border block or conversely.
\end{proof}

\begin{lemma}
\label{interiorlemma}
Let $B = \Gamma(m)$ be a commutative block partition.
Then the two commuting empty blocks must be interior blocks.
\end{lemma}

\begin{proof}
Let $R_1$, \dots, $R_\ell$ be the empty blocks from left to right along the north side of $I^2$.
It is clear from the geometric realization that neither associativity nor the interchange law
can change the order of $R_1$, \dots, $R_\ell$.
The same applies to the other three sides.
\end{proof}


\section{Commutative block partitions in arity 10}

\begin{lemma}
\label{arity10slices}
Let $B = \Gamma(m)$ be a commutative block partition of arity 10.
Then $B$ has at least two and at most four parallel slices in either direction (horizontal or vertical).
\end{lemma}

\begin{proof}
Lemma \ref{twomaincuts} shows that $B$ contains both main cuts; since $B$ contains 10 empty blocks,
$B$ has at most five parallel slices (four primary cuts) in either direction.
But if there are four primary cuts in one direction and the main cut in the other direction,
then we have 10 empty blocks, each of which is a border block, contradicting Lemma \ref{interiorlemma}.
\end{proof}

\begin{lemma}
\label{arity10blocks}
Let $B = \Gamma(m)$ be a commutative block partition of any arity.
Then $B$ has both main cuts by Lemma \ref{twomaincuts}, and hence $B$ consists of the union of
four square quarters $A_1$, \dots, $A_4$ (in the NW, NE,SW, SE corners respectively).
If one of these quarters has an empty block which is interior to $B$,
then that quarter contains at least three empty blocks.
If one of these quarters has two empty blocks which are both interior to $B$,
then that quarter contains at least four empty blocks.
Hence $B$ contains at least seven empty blocks.
\end{lemma}

\begin{proof}
If one of the four subrectangles has only two empty blocks then
these two blocks were created by a main cut, and hence both of them are border blocks in $B$.
Similarly, if one of the rectangles has only three empty blocks, then either these three blocks
are three parallel slices (in which case all three are border blocks in $B$) or these
three blocks were created by a main cut in one direction followed by the main cut in the
other direction in one of the blocks formed by the first main cut (in which case only one
of the three blocks is an interior block in $B$).
Lemma \ref{interiorlemma} shows that $B$ has at least two interior blocks, and these can
occur either in two different subrectangles or in the same subrectangle.
For different subrectangles, $B$ contains at least $3+3+1+1$ empty blocks, and for
the same subrectangle, $B$ contains at least $4+1+1+1$ empty blocks.
\end{proof}

\begin{proposition}
\label{atleast8proposition}
A commutative block partition $B$ has at least eight empty blocks.
\end{proposition}

\begin{proof}
Proposition \ref{twomaincuts} shows that $B$ must have both horizontal and vertical main cuts.
Lemma \ref{borderlemma} shows that an interior block cannot commute with a border block,
so $B$ must have at least two interior empty blocks.
The proof of Lemma \ref{arity10blocks} shows that the number of empty blocks in $B$ is at least seven,
with the minimum occurring if and only if there are two interior blocks in the same quarter $A_1$, \dots, $A_4$.
Reflecting in the horizontal and/or vertical axes if necessary, we may assume that the NW quarter $A_1$
contains two empty blocks which are interior to $B$ and contains only the horizontal main cut (otherwise
we reflect in the NW-SE diagonal).
Figure \ref{atleast8} shows the three partitions with seven empty blocks satisfying these conditions.

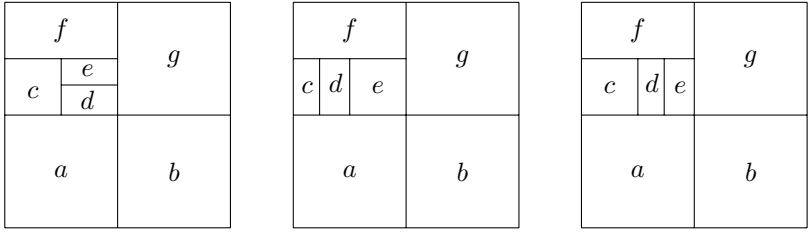
\begin{figure}[ht]
\begin{center}
\begin{tikzpicture}[ draw = black, x = 10 mm, y = 10 mm ]
\draw
(0,0) -- (3,0)
(0,1.5) -- (3,1.5)
(0,3) -- (3,3)
(0,2.25) -- (1.5,2.25)
(0.75,1.9) -- (1.5,1.9)
(0,0) -- (0,3)
(0.75,1.5) -- (0.75,2.25)
(1.5,0) -- (1.5,3)
(3,0) -- (3,3)
(0.75,0.75) node {$a$}
(2.25,0.75) node {$b$}
(1.1,2.1) node {$e$}
(0.375,1.8) node {$c$}
(1.1,1.7) node {$d$}
(0.75,2.625) node {$f$}
(2.25,2.25) node {$g$};
\end{tikzpicture}
\qquad
\begin{tikzpicture}[ draw = black, x = 10 mm, y = 10 mm ]
\draw
(0,0) -- (3,0)
(0,1.5) -- (3,1.5)
(0,3) -- (3,3)
(0,2.25) -- (1.5,2.25)
(0,0) -- (0,3)
(0.75,1.5) -- (0.75,2.25)
(0.35,1.5) -- (0.35,2.25)
(1.5,0) -- (1.5,3)
(3,0) -- (3,3)
(0.7500,0.750) node {$a$}
(2.2500,0.750) node {$b$}
(0.1875,1.875) node {$c$}
(0.5625,1.925) node {$d$}
(1.1250,1.875) node {$e$}
(0.7500,2.625) node {$f$}
(2.2500,2.250) node {$g$};
\end{tikzpicture}
\qquad
\begin{tikzpicture}[ draw = black, x = 10 mm, y = 10 mm ]
\draw
(0,0) -- (3,0)
(0,1.5) -- (3,1.5)
(0,3) -- (3,3)
(0,2.25) -- (1.5,2.25)
(0,0) -- (0,3)
(0.75,1.5) -- (0.75,2.25)
(1.1,1.5) -- (1.1,2.25)
(1.5,0) -- (1.5,3)
(3,0) -- (3,3)
(0.7500,0.750) node {$a$}
(2.2500,0.750) node {$b$}
(0.3750,1.875) node {$c$}
(0.9375,1.925) node {$d$}
(1.3125,1.875) node {$e$}
(0.7500,2.625) node {$f$}
(2.2500,2.25) node {$g$};
\end{tikzpicture}
\end{center}
\vspace{-4mm}
\caption{Three block partitions for proof of Proposition \ref{atleast8proposition}}
\label{atleast8}
\end{figure}

\noindent
Consider the monomial corresponding to the first partition in Figure \ref{atleast8}.
We may apply the interchange law only where two orthogonal cuts intersect at a point which is
interior to both; that is, at a plus $+$ configuration.
We may apply associativity only where we have $( - \ast - ) \ast -$ or $- \ast ( - \ast - )$.
At each step, there is only one possible rewriting that may be applied;
we underline the three (associativity) or four (interchange law) factors involved:
\begin{align*}
( \underline{a} \wedgehor \underline{b} )
\wedgever
( \underline{( ( c \wedgehor ( d \wedgever e ) ) \wedgever f )} \wedgehor \underline{g} )
&\equiv
( \underline{a} \wedgever ( \underline{( c \wedgehor ( d \wedgever e ) )} \wedgever \underline{f} ) )
\wedgehor
( b \wedgever g )
\\
&\equiv
( \underline{( a \wedgever ( c \wedgehor ( d \wedgever e ) ) )} \wedgever \underline{f} )
\wedgehor
( \underline{b} \wedgever \underline{g} )
\\
&\equiv
( a \wedgever ( c \wedgehor ( d \wedgever e ) ) \wedgehor b )
\wedgever
( f \wedgehor g ).
\end{align*}
Similar calculations apply to the second and third block partitions.
In this way, we have computed the entire equivalence class of the original monomial subject
to rewriting using associativity and the interchange law.
From this we see that no block partition with seven empty blocks admits a commutativity relation.
\end{proof}

The method used for the proof of Proposition \ref{atleast8proposition} can be extended
to show that a dyadic block partition with eight empty blocks cannot be commutative.
In this case, we have the following subcases:
(i)
one quarter $A_i$ has five empty blocks, and the other three are empty;
(ii)
one quarter $A_i$ has four empty blocks, another $A_j$ has two empty blocks,
and the other two are empty
(here we distinguish two subsubcases, depending on whether $A_i$ and $A_j$ share
an edge or only a corner);
(iii)
two quarters $A_i$ and $A_j$ each have three empty blocks, and the other two are empty
(with the same two subsubcases).
This provides a completely different proof, independent of machine computation, of one of
the main results in \cite{BM2016}; we omit the (rather lengthy) details.

In what follows, we write $B$ for a commutative block partition with 10 empty blocks.
Lemma \ref{interiorlemma} shows that the commuting blocks must be interior blocks, and
Lemma \ref{arity10slices} shows that $B$ has either two, three, or four parallel slices in either direction.
Thus, if $B$ has three (respectively four) parallel slices in one direction, then the commuting blocks must be
in the middle slice (respectively the middle two slices).
Without loss of generality, interchanging horizontal and vertical if necessary, we may assume that
these parallel slices are vertical.

\subsection{Four parallel vertical slices}

In this case we have the vertical and horizontal main cuts, and two additional vertical primary cuts.
Applying horizontal associativity if necessary, this gives the following configuration:
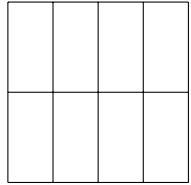
\[
\begin{tikzpicture} [ draw = black, x = 8 mm, y = 8 mm ]
\draw
(0.00,0.00) -- (0.00,3.00)
(0.75,0.00) -- (0.75,3.00)
(1.50,0.00) -- (1.50,3.00)
(2.25,0.00) -- (2.25,3.00)
(3.00,0.00) -- (3.00,3.00)
(0.00,0.00) -- (3.00,0.00)
(0.00,1.50) -- (3.00,1.50)
(0.00,3.00) -- (3.00,3.00);
\end{tikzpicture}
\]
This configuration has eight empty blocks, all of which are border blocks.
We need two more cuts to create two interior blocks.
Applying vertical associativity if necessary in the second slice from the left,
and applying a dihedral symmetry of the square if necessary,
we are left with three possible configurations:
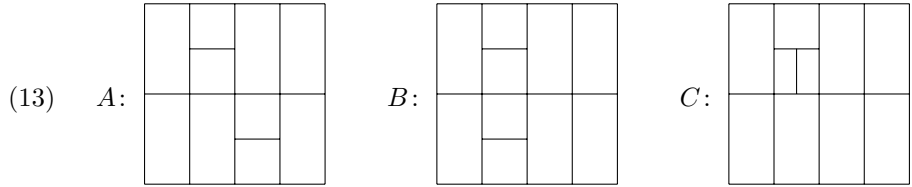
\begin{equation}
\label{configsABC}
A\colon
\adjustbox{valign=m}{
\begin{tikzpicture} [ draw = black, x = 8 mm, y = 8 mm ]
\draw
(0.00,0.00) -- (0.00,3.00)
(0.75,0.00) -- (0.75,3.00)
(1.50,0.00) -- (1.50,3.00)
(2.25,0.00) -- (2.25,3.00)
(3.00,0.00) -- (3.00,3.00)
(0.00,0.00) -- (3.00,0.00)
(0.00,1.50) -- (3.00,1.50)
(0.00,3.00) -- (3.00,3.00)
(0.75,2.25) -- (1.50,2.25)
(1.50,0.75) -- (2.25,0.75);
\end{tikzpicture}
}
\quad\quad
B\colon
\adjustbox{valign=m}{
\begin{tikzpicture} [ draw = black, x = 8 mm, y = 8 mm ]
\draw
(0.00,0.00) -- (0.00,3.00)
(0.75,0.00) -- (0.75,3.00)
(1.50,0.00) -- (1.50,3.00)
(2.25,0.00) -- (2.25,3.00)
(3.00,0.00) -- (3.00,3.00)
(0.00,0.00) -- (3.00,0.00)
(0.00,1.50) -- (3.00,1.50)
(0.00,3.00) -- (3.00,3.00)
(0.75,2.25) -- (1.50,2.25)
(0.75,0.75) -- (1.50,0.75);
\end{tikzpicture}
}
\quad\quad
C\colon
\adjustbox{valign=m}{
\begin{tikzpicture} [ draw = black, x = 8 mm, y = 8 mm ]
\draw
(0.000,0.00) -- (0.000,3.00)
(0.750,0.00) -- (0.750,3.00)
(1.500,0.00) -- (1.500,3.00)
(2.250,0.00) -- (2.250,3.00)
(3.000,0.00) -- (3.000,3.00)
(0.000,0.00) -- (3.000,0.00)
(0.000,1.50) -- (3.000,1.50)
(0.000,3.00) -- (3.000,3.00)
(0.750,2.25) -- (1.500,2.25)
(1.125,1.50) -- (1.125,2.25);
\end{tikzpicture}
}
\end{equation}

\subsubsection{Configuration $A$}

We present simultaneously the algebraic and geometric steps in the proof
of a new commutativity relation.
We label the empty blocks in the initial configuration as follows:
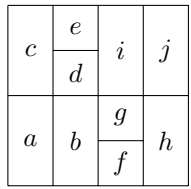
\[
\begin{tikzpicture}[ draw = black, x = 8 mm, y = 8 mm ]
\draw
(0,0) -- (3,0)
(0,1.5) -- (1.5,1.5)
(0,3) -- (3,3)
(0.75,2.25) -- (1.5,2.25)
(1.5,0.75) -- (2.25,0.75)
(0,0) -- (0,3)
(0.75,0) -- (0.75,3)
(1.5,0) -- (1.5,3)
(3,0) -- (3,3)
(2.25,0) -- (2.25,3)
(1.5,1.5) -- (3,1.5)
(0.375,0.75) node {$a$}
(1.125,0.75) node {$b$}
(1.875,0.375) node {$f$}
(1.875,1.15) node {$g$}
(2.625,0.75) node {$h$}
(0.375,2.25) node {$c$}
(1.125,1.875) node {$d$}
(1.125,2.625) node {$e$}
(1.875,2.25) node {$i$}
(2.625,2.25) node {$j$};
\end{tikzpicture}
\]
We show that this partition admits a commutativity relation transposing $d$ and $g$.
We refer the reader to Figure \ref{bigpicture} and \S\ref{diagramchasing} as an aid
to understanding the proof.
In the following list of monomials, we indicate the four factors $w, x, y, z$ taking part
in each application of the interchange law.
We omit parentheses in products using the same operation two or more times;
the factors $w, x, y, z$ make clear how we reassociate such products
between two consecutive applications of the interchange law.
The diagrams which appear after the list of monomials represent the same steps in
geometric form; in each application of the interchange law as the rewrite rule
$( a \star_2 b ) \star_1 ( c \star_2 d ) \mapsto ( a \star_1 c ) \star_2 ( b \star_1 d )$
where
$\{ \star_1, \star_2 \} = \{ \wedgehor, \wedgever \}$,
we indicate the root operation $\star_1$ by a thick line and the child operations $\star_2$ by dotted lines:
\begin{align}
&\text{factors}\; w, x, y, z & & \text{result of application of interchange law}
\notag
\\ \midrule
&\text{initial configuration}  &
&((a \wedgehor b)\wedgever (c \wedgehor (d \wedgever e)))\wedgehor (((f \wedgever g)\wedgehor h )\wedgever (i \wedgehor j))
\notag
\\[-2pt]
&f \wedgever g, h , i, j  &
&((a \wedgehor b)\wedgever (c \wedgehor (d \wedgever e)))\wedgehor ((f \wedgever g \wedgever i )\wedgehor (h \wedgever j))
\label{I1} \tag{I1}
\\[-2pt]
&f, g\wedgever i, h , j  &
&((a \wedgehor b)\wedgever (c \wedgehor (d \wedgever e)))\wedgehor ((f \wedgehor h )\wedgever ((g \wedgever i) \wedgehor j))
\label{I2} \tag{I2}
\\[-2pt]
&a, b, c, d \wedgever e  &
&((a \wedgever c)\wedgehor (b \wedgever d \wedgever e))\wedgehor ((f \wedgehor h )\wedgever ((g \wedgever i) \wedgehor j))
\label{I3} \tag{I3}
\\[-2pt]
&a, c, b \wedgever d, e  &
&((a \wedgehor (b \wedgever d))\wedgever (c \wedgehor e))\wedgehor ((f \wedgehor h )\wedgever ((g \wedgever i) \wedgehor j))
\label{I4} \tag{I4}
\\[-2pt]
&a \wedgehor (b \wedgever d), c \wedgehor e, f \wedgehor h, (g \wedgever i) \wedgehor j  &
&((a \wedgehor (b \wedgever d))\wedgehor (f \wedgehor h ))\wedgever (c \wedgehor e \wedgehor (g \wedgever i) \wedgehor j)
\label{I5} \tag{I5}
\\[-2pt]
&a \wedgehor (b \wedgever d), f \wedgehor h , c \wedgehor e \wedgehor(g \wedgever i), j  &
&((a \wedgehor (b \wedgever d))\wedgever (c \wedgehor e\wedgehor(g \wedgever i)))\wedgehor ((f \wedgehor h ) \wedgever j)
\label{I6} \tag{I6}
\\[-2pt]
&a, b \wedgever d, c \wedgehor e, g \wedgever i  &
&((a \wedgever (c \wedgehor e))\wedgehor (b \wedgever d \wedgever g \wedgever i))\wedgehor ((f \wedgehor h ) \wedgever j)
\label{I7} \tag{I7}
\\[-2pt]
&a, c \wedgehor e, b \wedgever d \wedgever g, i  &
&((a \wedgehor (b \wedgever d \wedgever g))\wedgever (c \wedgehor e \wedgehor i))\wedgehor ((f \wedgehor h ) \wedgever j))
\label{I8} \tag{I8}
\\[-2pt]
&a \wedgehor (b \wedgever d \wedgever g), c \wedgehor e \wedgehor i, f \wedgehor h , j  &
&((a \wedgehor (b \wedgever d \wedgever g)\wedgehor (f \wedgehor h ) )\wedgever (c \wedgehor e \wedgehor i \wedgehor j)
\label{I9} \tag{I9}
\\[-2pt]
&a \wedgehor (b \wedgever d \wedgever g), f \wedgehor h, c \wedgehor e, i\wedgehor j  &
&((a \wedgehor (b \wedgever d \wedgever g))\wedgever (c \wedgehor e) )\wedgehor ((f \wedgehor h ) \wedgever (i \wedgehor j))
\label{I10} \tag{I10}
\\[-2pt]
&a, b \wedgever d \wedgever g, c, e  &
&((a \wedgever c)\wedgehor (b \wedgever d \wedgever g\wedgever e )\wedgehor ((f \wedgehor h ) \wedgever (i \wedgehor j))
\label{I11} \tag{I11}
\\[-2pt]
&a, c, b \wedgever d, g\wedgever e  &
&((a \wedgehor(b \wedgever d ))\wedgever (c \wedgehor (g\wedgever e) )\wedgehor ((f \wedgehor h ) \wedgever (i \wedgehor j))
\label{I12} \tag{I12}
\\[-2pt]
&a \wedgehor(b \wedgever d ), c \wedgehor (g\wedgever e), f \wedgehor h, i \wedgehor j  &
&(a \wedgehor(b \wedgever d )\wedgehor (f \wedgehor h ) )\wedgever (( c \wedgehor (g\wedgever e) \wedgehor (i \wedgehor j))
\label{I13} \tag{I13}
\\[-2pt]
&a, (b \wedgever d )\wedgehor f \wedgehor h, c \wedgehor (g\wedgever e), i \wedgehor j  &
&(a \wedgever ( c \wedgehor (g \wedgever e))\wedgehor (((b \wedgever d )\wedgehor f \wedgehor h) \wedgever (i \wedgehor j))
\label{I14} \tag{I14}
\\[-2pt]
&b \wedgever d, f \wedgehor h, i , j  &
&(a \wedgever ( c \wedgehor (g \wedgever e))\wedgehor ((b \wedgever d \wedgever i) \wedgehor ((f \wedgehor h) \wedgever j))
\label{I15} \tag{I15}
\\[-2pt]
& b, d \wedgever i, f \wedgehor h, j  &
&(a \wedgever ( c \wedgehor (g \wedgever e))\wedgehor ((b \wedgehor f \wedgehor h) \wedgever ((d \wedgever i) \wedgehor j))
\label{I16} \tag{I16}
\\[-2pt]
&a, c \wedgehor (g \wedgever e), b \wedgehor f \wedgehor h, (d \wedgever i) \wedgehor j  &
&(a \wedgehor b \wedgehor f \wedgehor h)\wedgever ((c \wedgehor (g \wedgever e)) \wedgehor ((d \wedgever i) \wedgehor j))
\label{I17} \tag{I17}
\\[-2pt]
&a \wedgehor b, f \wedgehor h, c \wedgehor (g \wedgever e),(d \wedgever i) \wedgehor j  &
&((a \wedgehor b) \wedgever(c \wedgehor (g \wedgever e)) \wedgehor ((f \wedgehor h) \wedgever ((d \wedgever i) \wedgehor j))
\label{I18} \tag{I18}
\\[-2pt]
&f, h, d \wedgever i, j  &
&((a \wedgehor b) \wedgever(c \wedgehor (g \wedgever e)) \wedgehor ((f \wedgever d \wedgever i) \wedgehor (h \wedgever j))
\label{I19} \tag{I19}
\\[-2pt]
&f \wedgever d, i,h, j  &
&((a \wedgehor b) \wedgever(c \wedgehor (g \wedgever e)) \wedgehor (((f \wedgever d) \wedgehor h) \wedgever (i \wedgehor j))
\label{I20} \tag{I20}
\end{align}
The same sequence of rewritings has the following geometric representation:
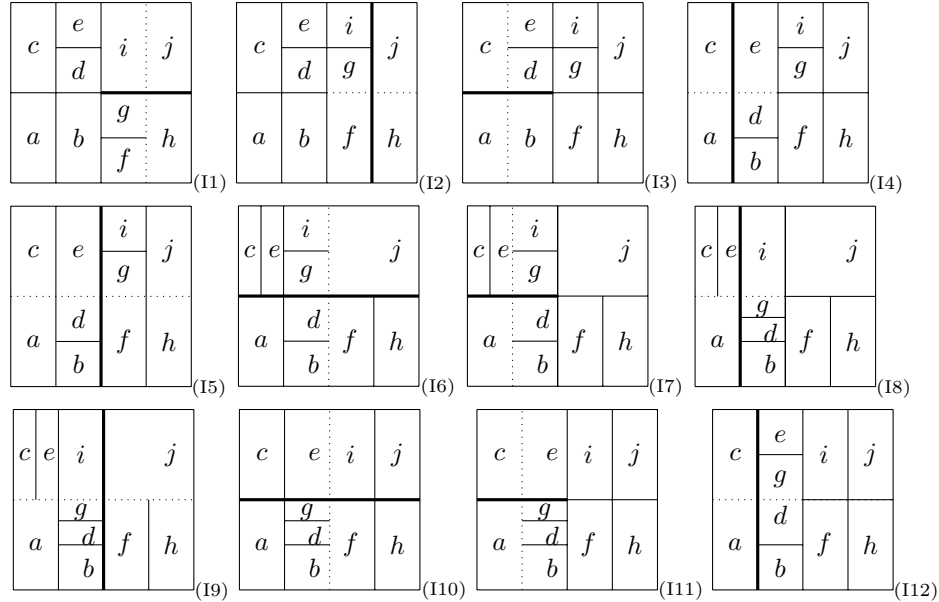
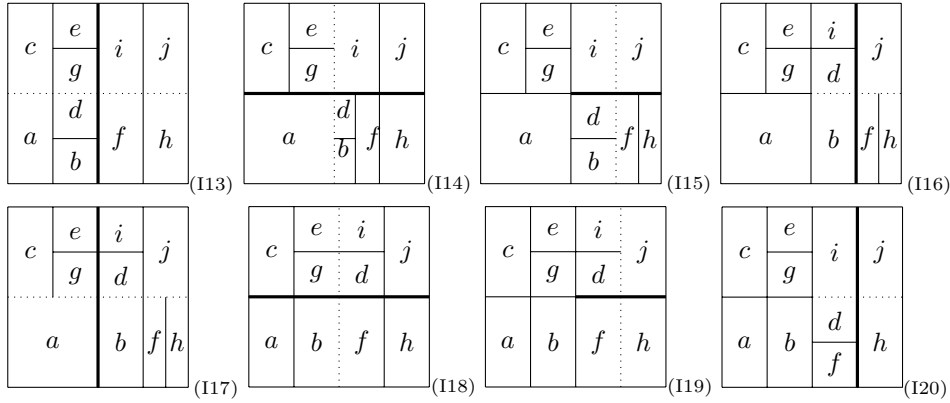
\begin{align*}
&
\adjustbox{valign=m}{$
\begin{tikzpicture}[ draw = black, x = 8 mm, y = 8 mm ]
\draw
(0,0) -- (3,0)
(0,1.5) -- (1.5,1.5)
(0,3) -- (3,3)
(0.75,2.25) -- (1.5,2.25)
(1.5,0.75) -- (2.25,0.75)
(0,0) -- (0,3)
(0.75,0) -- (0.75,3)
(1.5,0) -- (1.5,3)
(3,0) -- (3,3)
(0.375,0.75) node {$a$}
(1.125,0.75) node {$b$}
(1.875,0.375) node {$f$}
(1.875,1.15) node {$g$}
(2.625,0.75) node {$h$}
(0.375,2.25) node {$c$}
(1.125,1.875) node {$d$}
(1.125,2.625) node {$e$}
(1.875,2.25) node {$i$}
(2.625,2.25) node {$j$};
\draw[dotted]
(2.25,0) -- (2.25,3);
\draw[ very thick]
(1.5,1.5) -- (3,1.5);
\end{tikzpicture}_{\eqref{I1}}
$}
\;
\adjustbox{valign=m}{$
\begin{tikzpicture}[ draw = black, x = 8 mm, y = 8 mm ]
\draw
(0,0) -- (3,0)
(0,1.5) -- (1.5,1.5)
(0,3) -- (3,3)
(0.75,2.25) -- (1.5,2.25)
(1.5,2.25) -- (2.25,2.25)
(0,0) -- (0,3)
(0.75,0) -- (0.75,3)
(1.5,0) -- (1.5,3)
(3,0) -- (3,3)
(0.375,0.75) node {$a$}
(1.125,0.75) node {$b$}
(1.875,0.75) node {$f$}
(1.875,1.875) node {$g$}
(2.625,0.75) node {$h$}
(0.375,2.25) node {$c$}
(1.125,1.875) node {$d$}
(1.125,2.625) node {$e$}
(1.875,2.625) node {$i$}
(2.625,2.25) node {$j$};
\draw[ very thick]
(2.25,0) -- (2.25,3);
\draw[dotted]
(1.5,1.5) -- (3,1.5);
\end{tikzpicture}_{\eqref{I2}}
$}
\;
\adjustbox{valign=m}{$
\begin{tikzpicture}[ draw = black, x = 8 mm, y = 8 mm ]
\draw
(0,0) -- (3,0)
(0,3) -- (3,3)
(0.75,2.25) -- (1.5,2.25)
(1.5,2.25) -- (2.25,2.25)
(2.25,0) -- (2.25,3)
(1.5,1.5) -- (3,1.5)
(0,0) -- (0,3)
(1.5,0) -- (1.5,3)
(3,0) -- (3,3)
(0.375,0.75) node {$a$}
(1.125,0.75) node {$b$}
(1.875,0.75) node {$f$}
(1.875,1.875) node {$g$}
(2.625,0.75) node {$h$}
(0.375,2.25) node {$c$}
(1.125,1.875) node {$d$}
(1.125,2.625) node {$e$}
(1.875,2.625) node {$i$}
(2.625,2.25) node {$j$};
\draw[dotted]
(0.75,0) -- (0.75,3);
\draw[very thick]
(0,1.5) -- (1.5,1.5);
\end{tikzpicture}_{\eqref{I3}}
$}
\;
\adjustbox{valign=m}{$
\begin{tikzpicture}[ draw = black, x = 8 mm, y = 8 mm ]
\draw
(0,0) -- (3,0)
(0,3) -- (3,3)
(0.75,0.75) -- (1.5,0.75)
(1.5,2.25) -- (2.25,2.25)
(0,0) -- (0,3)
(1.5,1.5) -- (3,1.5)
(2.25,0) -- (2.25,3)
(1.5,0) -- (1.5,3)
(3,0) -- (3,3)
(0.375,0.75) node {$a$}
(1.125,0.375) node {$b$}
(1.875,0.75) node {$f$}
(1.875,1.875) node {$g$}
(2.625,0.75) node {$h$}
(0.375,2.25) node {$c$}
(1.125,1.125) node {$d$}
(1.125,2.25) node {$e$}
(1.875,2.625) node {$i$}
(2.625,2.25) node {$j$};
\draw[dotted]
(0,1.5) -- (1.5,1.5) ;
\draw[ very thick]
(0.75,0) -- (0.75,3);
\end{tikzpicture}_{\eqref{I4}}
$}
\\
&
\adjustbox{valign=m}{$
\begin{tikzpicture}[ draw = black, x = 8 mm, y = 8 mm ]
\draw
(0,0) -- (3,0)
(0,3) -- (3,3)
(0.75,0.75) -- (1.5,0.75)
(1.5,2.25) -- (2.25,2.25)
(0,0) -- (0,3)
(2.25,0) -- (2.25,3)
(0.75,0) -- (0.75,3)
(3,0) -- (3,3)
(0.375,0.75) node {$a$}
(1.125,0.375) node {$b$}
(1.875,0.75) node {$f$}
(1.875,1.875) node {$g$}
(2.625,0.75) node {$h$}
(0.375,2.25) node {$c$}
(1.125,1.125) node {$d$}
(1.125,2.25) node {$e$}
(1.875,2.625) node {$i$}
(2.625,2.25) node {$j$};
\draw[dotted]
(0,1.5) -- (3,1.5) ;
\draw[ very thick]
(1.5,0) -- (1.5,3);
\end{tikzpicture}_{\eqref{I5}}
$}
\;
\adjustbox{valign=m}{$
\begin{tikzpicture}[ draw = black, x = 8 mm, y = 8 mm ]
\draw
(0,0) -- (3,0)
(0,3) -- (3,3)
(0.75,0.75) -- (1.5,0.75)
(0.75,2.25) -- (1.5,2.25)
(0.75,0) -- (0.75,1.5)
(0,0) -- (0,3)
(2.25,0) -- (2.25,1.5)
(3,0) -- (3,3)
(0.75,1.5) -- (0.75,3)
(0.375,1.5) -- (0.375,3)
(0.375,0.75) node {$a$}
(1.25,0.375) node {$b$}
(1.85,0.75) node {$f$}
(1.125,1.875) node {$g$}
(2.625,0.75) node {$h$}
(0.185,2.25) node {$c$}
(1.25,1.125) node {$d$}
(0.6,2.25) node {$e$}
(1.125,2.625) node {$i$}
(2.625,2.25) node {$j$};
\draw[very thick]
(0,1.5) -- (3,1.5);
\draw[dotted]
(1.5,0) -- (1.5,3);
\end{tikzpicture}_{\eqref{I6}}
$}
\;
\adjustbox{valign=m}{$
\begin{tikzpicture}[ draw = black, x = 8 mm, y = 8 mm ]
\draw
(0,0) -- (3,0)
(0,3) -- (3,3)
(1.5,1.5) -- (3,1.5)
(0.75,0.75) -- (1.5,0.75)
(0.75,2.25) -- (1.5,2.25)
(0,0) -- (0,3)
(0.375,1.5) -- (0.375,3)
(1.5,0) -- (1.5,3)
(1.5,0) -- (1.5,3)
(2.25,0) -- (2.25,1.5)
(3,0) -- (3,3)
(0.375,0.75) node {$a$}
(1.25,0.375) node {$b$}
(1.85,0.75) node {$f$}
(1.125,1.875) node {$g$}
(2.625,0.75) node {$h$}
(0.185,2.25) node {$c$}
(1.25,1.125) node {$d$}
(0.6,2.25) node {$e$}
(1.125,2.625) node {$i$}
(2.625,2.25) node {$j$};
\draw[ very thick]
(0,1.5) -- (1.5,1.5);
\draw[ dotted]
(0.75,0) -- (0.75,3);
\end{tikzpicture}_{\eqref{I7}}
$}
\;
\adjustbox{valign=m}{$
\begin{tikzpicture}[ draw = black, x = 8 mm, y = 8 mm ]
\draw
(0,0) -- (3,0)
(0,3) -- (3,3)
(0.75,0.75) -- (1.5,0.75)
(0.75,1.15) -- (1.5,1.15)
(0,0) -- (0,3)
(0.375,1.5) -- (0.375,3)
(1.5,0) -- (1.5,3)
(2.25,0) -- (2.25,1.5)
(3,0) -- (3,3)
(1.5,0) -- (1.5,3)
(1.5,1.5) -- (3,1.5)
(0.375,0.75) node {$a$}
(1.25,0.375) node {$b$}
(1.85,0.75) node {$f$}
(1.125,1.3) node {$g$}
(2.625,0.75) node {$h$}
(0.185,2.25) node {$c$}
(1.25,0.9) node {$d$}
(0.6,2.25) node {$e$}
(1.125,2.25) node {$i$}
(2.625,2.25) node {$j$};
\draw[dotted]
(0,1.5) -- (1.5,1.5);
\draw[ very thick]
(0.75,0) -- (0.75,3);
\end{tikzpicture}_{\eqref{I8}}
$}
\\
&
\adjustbox{valign=m}{$
\begin{tikzpicture}[ draw = black, x = 8 mm, y = 8 mm ]
\draw
(0,0) -- (3,0)
(0,3) -- (3,3)
(0.75,0.75) -- (1.5,0.75)
(0.75,1.15) -- (1.5,1.15)
(0,0) -- (0,3)
(0.375,1.5) -- (0.375,3)
(0.75,1.5) -- (0.75,3)
(0.75,0) -- (0.75,1.5)
(2.25,0) -- (2.25,1.5)
(3,0) -- (3,3)
(0.375,0.75) node {$a$}
(1.25,0.375) node {$b$}
(1.85,0.75) node {$f$}
(1.125,1.3) node {$g$}
(2.625,0.75) node {$h$}
(0.185,2.25) node {$c$}
(1.25,0.9) node {$d$}
(0.6,2.25) node {$e$}
(1.125,2.25) node {$i$}
(2.625,2.25) node {$j$};
\draw[very thick]
(1.5,0) -- (1.5,3);
\draw[  dotted]
(0,1.5) -- (3,1.5);
\end{tikzpicture}_{\eqref{I9}}
$}
\;
\adjustbox{valign=m}{$
\begin{tikzpicture}[ draw = black, x = 8 mm, y = 8 mm ]
\draw
(0,0) -- (3,0)
(0,3) -- (3,3)
(0.75,0.75) -- (1.5,0.75)
(0.75,1.15) -- (1.5,1.15)
(0,0) -- (0,3)
(0.75,0) -- (0.75,3)
(2.25,0) -- (2.25,3)
(3,0) -- (3,3)
(0.375,0.75) node {$a$}
(1.25,0.375) node {$b$}
(1.85,0.75) node {$f$}
(1.125,1.3) node {$g$}
(2.625,0.75) node {$h$}
(0.375,2.25) node {$c$}
(1.25,0.9) node {$d$}
(1.25,2.25) node {$e$}
(1.85,2.25) node {$i$}
(2.625,2.25) node {$j$};
\draw[dotted]
(1.5,0) -- (1.5,3);
\draw[very thick]
(0,1.5) -- (3,1.5);
\end{tikzpicture}_{\eqref{I10}}
$}
\;
\adjustbox{valign=m}{$
\begin{tikzpicture}[ draw = black, x = 8 mm, y = 8 mm ]
\draw
(0,0) -- (3,0)
(0,3) -- (3,3)
(1.5,1.5) -- (3,1.5)
(0.75,0.75) -- (1.5,0.75)
(0.75,1.15) -- (1.5,1.15)
(0,0) -- (0,3)
(1.5,0) -- (1.5,3)
(2.25,0) -- (2.25,3)
(3,0) -- (3,3)
(0.375,0.75) node {$a$}
(1.25,0.375) node {$b$}
(1.85,0.75) node {$f$}
(1.125,1.3) node {$g$}
(2.625,0.75) node {$h$}
(0.375,2.25) node {$c$}
(1.25,0.9) node {$d$}
(1.25,2.25) node {$e$}
(1.85,2.25) node {$i$}
(2.625,2.25) node {$j$};
\draw[dotted]
(0.75,0) -- (0.75,3);
\draw[very thick]
(0,1.5) -- (1.5,1.5);
\end{tikzpicture}_{\eqref{I11}}
$}
\;
\adjustbox{valign=m}{$
\begin{tikzpicture}[ draw = black, x = 8 mm, y = 8 mm ]
\draw
(0,0) -- (3,0)
(0,3) -- (3,3)
(1.5,1.5) -- (3,1.5)
(0.75,0.75) -- (1.5,0.75)
(0.75,2.25) -- (1.5,2.25)
(0,0) -- (0,3)
(0.75,0) -- (0.75,3)
(2.25,0) -- (2.25,3)
(3,0) -- (3,3)
(1.5,0) -- (1.5,3)
(0.375,0.75) node {$a$}
(1.125,0.375) node {$b$}
(1.85,0.75) node {$f$}
(1.125,1.85) node {$g$}
(2.625,0.75) node {$h$}
(0.375,2.25) node {$c$}
(1.125,1.25) node {$d$}
(1.125,2.555) node {$e$}
(1.85,2.25) node {$i$}
(2.625,2.25) node {$j$};
\draw[very thick]
(0.75,0) -- (0.75,3);
\draw[dotted]
(0,1.5) -- (3,1.5);
\end{tikzpicture}_{\eqref{I12}}
$}
\\
&
\adjustbox{valign=m}{$
\begin{tikzpicture}[ draw = black, x = 8 mm, y = 8 mm ]
\draw
(0,0) -- (3,0)
(0,3) -- (3,3)
(0.75,0.75) -- (1.5,0.75)
(0.75,2.25) -- (1.5,2.25)
(0,0) -- (0,3)
(0.75,0) -- (0.75,3)
(2.25,0) -- (2.25,3)
(3,0) -- (3,3)
(0.375,0.75) node {$a$}
(1.125,0.375) node {$b$}
(1.85,0.75) node {$f$}
(1.125,1.85) node {$g$}
(2.625,0.75) node {$h$}
(0.375,2.25) node {$c$}
(1.125,1.25) node {$d$}
(1.125,2.555) node {$e$}
(1.85,2.25) node {$i$}
(2.625,2.25) node {$j$};
\draw[very thick]
(1.5,0) -- (1.5,3);
\draw[dotted]
(0,1.5) -- (3,1.5);
\end{tikzpicture}_{\eqref{I13}}
$}
\;
\adjustbox{valign=m}{$
\begin{tikzpicture}[ draw = black, x = 8 mm, y = 8 mm ]
\draw
(0,0) -- (3,0)
(0,3) -- (3,3)
(1.5,0.75) -- (1.85,0.75)
(0.75,2.25) -- (1.5,2.25)
(0,0) -- (0,3)
(0.75,1.5) -- (0.75,3)
(2.25,0) -- (2.25,1.5)
(3,0) -- (3,3)
(1.85,0) -- (1.85,1.5)
(2.25,1.5) -- (2.25,3)
(1.5,1.5) -- (3,1.5)
(0.75,0.75) node {$a$}
(1.65,0.6) node {$b$}
(2.15,0.75) node {$f$}
(1.125,1.85) node {$g$}
(2.625,0.75) node {$h$}
(0.375,2.25) node {$c$}
(1.65,1.25) node {$d$}
(1.125,2.555) node {$e$}
(1.85,2.25) node {$i$}
(2.625,2.25) node {$j$};
\draw[dotted]
(1.5,0) -- (1.5,3);
\draw[ very thick]
(0,1.5) -- (3,1.5);
\end{tikzpicture}_{\eqref{I14}}
$}
\;
\adjustbox{valign=m}{$
\begin{tikzpicture}[ draw = black, x = 8 mm, y = 8 mm ]
\draw
(0,0) -- (3,0)
(0,3) -- (3,3)
(0,0) -- (0,3)
(3,0) -- (3,3)
(1.5,0) -- (1.5,3)
(1.5,0.75) -- (2.25,0.75)
(0,1.5) -- (3,1.5)
(0.75,2.25) -- (1.5,2.25)
(0.75,1.5) -- (0.75,3)
(2.625,0) -- (2.625,1.5)
(0.75,0.75) node {$a$}
(1.875,0.375) node {$b$}
(0.375,2.25) node {$c$}
(1.875,1.125) node {$d$}
(1.125,2.555) node {$e$}
(2.4375,0.75) node {$f$}
(1.125,1.85) node {$g$}
(2.8125,0.75) node {$h$}
(1.875,2.25) node {$i$}
(2.625,2.25) node {$j$};
\draw[dotted]
(2.25,0) -- (2.25,3);
\draw[ very thick]
(1.5,1.5) -- (3,1.5);
\end{tikzpicture}_{\eqref{I15}}
$}
\;
\adjustbox{valign=m}{$
\begin{tikzpicture}[ draw = black, x = 8 mm, y = 8 mm ]
\draw
(0,0) -- (3,0)
(0,3) -- (3,3)
(1.5,2.25) -- (2.25,2.25)
(0.75,2.25) -- (1.5,2.25)
(0,0) -- (0,3)
(0.75,1.5) -- (0.75,3)
(1.5,0) -- (1.5,3)
(3,0) -- (3,3)
(2.625,0) -- (2.625,1.5)
(0,1.5) -- (1.5,1.5)
(0.75,0.75) node {$a$}
(1.875,0.75) node {$b$}
(2.4375,0.75) node {$f$}
(1.125,1.85) node {$g$}
(2.8125,0.75) node {$h$}
(0.375,2.25) node {$c$}
(1.875,1.85) node {$d$}
(1.125,2.555) node {$e$}
(1.85,2.555) node {$i$}
(2.625,2.25) node {$j$};
\draw[dotted]
(1.5,1.5) -- (3,1.5);
\draw[ very thick]
(2.25,0) -- (2.25,3);
\end{tikzpicture}_{\eqref{I16}}
$}
\\
&
\adjustbox{valign=m}{$
\begin{tikzpicture}[ draw = black, x = 8 mm, y = 8 mm ]
\draw
(0,0) -- (3,0)
(0,3) -- (3,3)
(1.5,2.25) -- (2.25,2.25)
(0.75,2.25) -- (1.5,2.25)
(0,0) -- (0,3)
(0.75,1.5) -- (0.75,3)
(1.5,0) -- (1.5,3)
(3,0) -- (3,3)
(2.625,0) -- (2.625,1.5)
(2.25,0) -- (2.25,3)
(0.75,0.75) node {$a$}
(1.875,0.75) node {$b$}
(2.4375,0.75) node {$f$}
(1.125,1.85) node {$g$}
(2.8125,0.75) node {$h$}
(0.375,2.25) node {$c$}
(1.875,1.85) node {$d$}
(1.125,2.555) node {$e$}
(1.85,2.555) node {$i$}
(2.625,2.25) node {$j$};
\draw[dotted]
(0,1.5) -- (3,1.5);
\draw[ very thick]
(1.5,0) -- (1.5,3);
\end{tikzpicture}_{\eqref{I17}}
$}
\;
\adjustbox{valign=m}{$
\begin{tikzpicture}[ draw = black, x = 8 mm, y = 8 mm ]
\draw
(0,0) -- (3,0)
(0,1.5) -- (1.5,1.5)
(0,3) -- (3,3)
(0.75,2.25) -- (1.5,2.25)
(1.5,2.25) -- (2.25,2.25)
(0,0) -- (0,3)
(0.75,0) -- (0.75,3)
(2.25,0) -- (2.25,3)
(3,0) -- (3,3)
(0.375,0.75) node {$a$}
(1.125,0.75) node {$b$}
(1.875,0.75) node {$f$}
(1.875,1.875) node {$d$}
(2.625,0.75) node {$h$}
(0.375,2.25) node {$c$}
(1.125,1.875) node {$g$}
(1.125,2.625) node {$e$}
(1.875,2.625) node {$i$}
(2.625,2.25) node {$j$};
\draw[dotted]
(1.5,0) -- (1.5,3);
\draw[ very thick]
(0,1.5) -- (3,1.5);
\end{tikzpicture}_{\eqref{I18}}
$}
\;
\adjustbox{valign=m}{$
\begin{tikzpicture}[ draw = black, x = 8 mm, y = 8 mm ]
\draw
(0,0) -- (3,0)
(0,3) -- (3,3)
(0.75,2.25) -- (1.5,2.25)
(1.5,2.25) -- (2.25,2.25)
(0,0) -- (0,3)
(0.75,0) -- (0.75,3)
(1.5,0) -- (1.5,3)
(3,0) -- (3,3)
(0,1.5) -- (1.5,1.5)
(0.375,0.75) node {$a$}
(1.125,0.75) node {$b$}
(1.875,0.75) node {$f$}
(1.875,1.875) node {$d$}
(2.625,0.75) node {$h$}
(0.375,2.25) node {$c$}
(1.125,1.875) node {$g$}
(1.125,2.625) node {$e$}
(1.875,2.625) node {$i$}
(2.625,2.25) node {$j$};
\draw[dotted]
(2.25,0) -- (2.25,3);
\draw[ very thick]
(1.5,1.5) -- (3,1.5);
\end{tikzpicture}_{\eqref{I19}}
$}
\;
\adjustbox{valign=m}{$
\begin{tikzpicture}[ draw = black, x = 8 mm, y = 8 mm ]
\draw
(0,0) -- (3,0)
(0,1.5) -- (1.5,1.5)
(0,3) -- (3,3)
(2.25,0) -- (2.25,3)
(1.5,0.75) -- (2.25,0.75)
(0.75,2.25) -- (1.5,2.25)
(0,0) -- (0,3)
(0.75,0) -- (0.75,3)
(1.5,0) -- (1.5,3)
(3,0) -- (3,3)
(0.375,0.75) node {$a$}
(1.125,0.75) node {$b$}
(1.875,0.375) node {$f$}
(1.875,1.1) node {$d$}
(2.625,0.75) node {$h$}
(0.375,2.25) node {$c$}
(1.125,1.875) node {$g$}
(1.125,2.625) node {$e$}
(1.875,2.25) node {$i$}
(2.625,2.25) node {$j$};
\draw[dotted]
(1.5,1.5) -- (3,1.5);
\draw[ very thick]
(2.25,0) -- (2.25,3);
\end{tikzpicture}_{\eqref{I20}}
$}
\end{align*}

\begin{theorem}
In every double interchange semigroup, the following commutativity relation holds
for all values of the arguments $a, \dots, j$:
\begin{align*}
&
((a \wedgehor b)\wedgever (c \wedgehor (d \wedgever e)))\wedgehor
(((f \wedgever g)\wedgehor h )\wedgever (i \wedgehor j))
\equiv {}
\\
&
((a \wedgehor b)\wedgever (c \wedgehor (g \wedgever e)))\wedgehor
(((f \wedgever d)\wedgehor h )\wedgever (i \wedgehor j))
\end{align*}
\end{theorem}

\subsubsection{Configuration $B$}

For configuration $B$ in display \eqref{configsABC}, we label only the two blocks which transpose
in the commutativity relation.
The required applications of associativity and the interchange law can easily be reconstructed
from the diagrams:
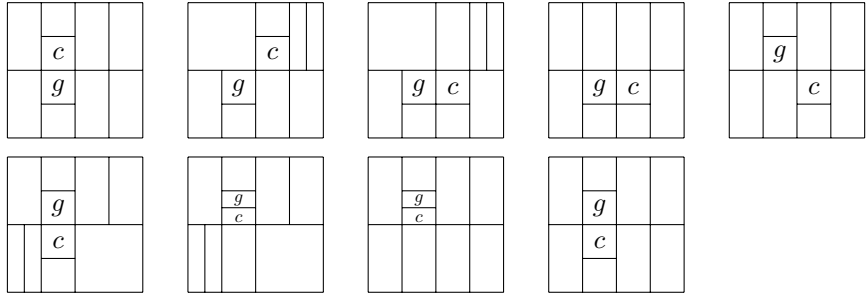
\begin{align*}
&
\adjustbox{valign=m}{
\begin{tikzpicture} [ draw = black, x = 6 mm, y = 6 mm ]
\draw
(0.00,0.00) -- (0.00,3.00)
(0.75,0.00) -- (0.75,3.00)
(1.50,0.00) -- (1.50,3.00)
(2.25,0.00) -- (2.25,3.00)
(3.00,0.00) -- (3.00,3.00)
(0.00,0.00) -- (3.00,0.00)
(0.00,1.50) -- (3.00,1.50)
(0.00,3.00) -- (3.00,3.00)
(0.75,2.25) -- (1.50,2.25)
(0.75,0.75) -- (1.50,0.75);
\draw
(1.125,1.875) node {$c$}
(1.125,1.1) node {$g$};
\end{tikzpicture}
}
\quad
\adjustbox{valign=m}{
\begin{tikzpicture} [ draw = black, x = 6 mm, y = 6 mm ]
\draw
(0.00,0.00) -- (0.00,3.00)
(0.75,0.00) -- (0.75,1.50)
(1.50,0.00) -- (1.50,3.00)
(2.25,0.00) -- (2.25,3.00)
(3.00,0.00) -- (3.00,3.00)
(0.00,0.00) -- (3.00,0.00)
(0.00,1.50) -- (3.00,1.50)
(0.00,3.00) -- (3.00,3.00)
(1.50,2.25) -- (2.25,2.25)
(0.75,0.75) -- (1.50,0.75)
(2.625,1.50) --(2.625,3.00);
\draw
(1.875,1.875) node {$c$}
(1.125,1.1) node {$g$};
\end{tikzpicture}
}
\quad
\adjustbox{valign=m}{
\begin{tikzpicture} [ draw = black, x = 6 mm, y = 6 mm ]
\draw
(0.00,0.00) -- (0.00,3.00)
(0.75,0.00) -- (0.75,1.50)
(1.50,0.00) -- (1.50,3.00)
(2.25,0.00) -- (2.25,3.00)
(3.00,0.00) -- (3.00,3.00)
(0.00,0.00) -- (3.00,0.00)
(0.00,1.50) -- (3.00,1.50)
(0.00,3.00) -- (3.00,3.00)
(1.50,0.750) -- (2.25,0.750)
(0.75,0.75) -- (1.50,0.75)
(2.625,1.50) --(2.625,3.00);
\draw
(1.875,1.1) node {$c$}
(1.125,1.1) node {$g$};
\end{tikzpicture}
}
\quad
\adjustbox{valign=m}{
\begin{tikzpicture} [ draw = black, x = 6 mm, y = 6 mm ]
\draw
(0.00,0.00) -- (0.00,3.00)
(0.75,0.00) -- (0.75,3.00)
(1.50,0.00) -- (1.50,3.00)
(2.25,0.00) -- (2.25,3.00)
(3.00,0.00) -- (3.00,3.00)
(0.00,0.00) -- (3.00,0.00)
(0.00,1.50) -- (3.00,1.50)
(0.00,3.00) -- (3.00,3.00)
(1.50,0.750) -- (2.25,0.750)
(0.75,0.75) -- (1.50,0.75);
\draw
(1.875,1.1) node {$c$}
(1.125,1.1) node {$g$};
\end{tikzpicture}
}
\quad
\adjustbox{valign=m}{
\begin{tikzpicture} [ draw = black, x = 6 mm, y = 6 mm ]
\draw
(0.00,0.00) -- (0.00,3.00)
(0.75,0.00) -- (0.75,3.00)
(1.50,0.00) -- (1.50,3.00)
(2.25,0.00) -- (2.25,3.00)
(3.00,0.00) -- (3.00,3.00)
(0.00,0.00) -- (3.00,0.00)
(0.00,1.50) -- (3.00,1.50)
(0.00,3.00) -- (3.00,3.00)
(0.75,2.25) -- (1.50,2.25)
(1.50,0.75) -- (2.25,0.75);
\draw
(1.875,1.1) node {$c$}
(1.125,1.875) node {$g$};
\end{tikzpicture}
}
\\[1mm]
&
\adjustbox{valign=m}{
\begin{tikzpicture} [ draw = black, x = 6 mm, y = 6 mm ]
\draw
(0.00,0.00) -- (0.00,3.00)
(0.75,0.00) -- (0.75,3.00)
(1.50,0.00) -- (1.50,3.00)
(2.25,1.50) -- (2.25,3.00)
(3.00,0.00) -- (3.00,3.00)
(0.00,0.00) -- (3.00,0.00)
(0.00,1.50) -- (3.00,1.50)
(0.00,3.00) -- (3.00,3.00)
(0.75,2.25) -- (1.50,2.25)
(0.75,0.75) -- (1.50,0.75)
(0.375,0.00) --(0.375,1.50);
\draw
(1.125,1.875) node {$g$}
(1.125,1.1) node {$c$};
\end{tikzpicture}
}
\quad
\adjustbox{valign=m}{
\begin{tikzpicture} [ draw = black, x = 6 mm, y = 6 mm ]
\draw
(0.00,0.00) -- (0.00,3.00)
(0.75,0.00) -- (0.75,3.00)
(1.50,0.00) -- (1.50,3.00)
(2.25,1.50) -- (2.25,3.00)
(3.00,0.00) -- (3.00,3.00)
(0.00,0.00) -- (3.00,0.00)
(0.00,1.50) -- (3.00,1.50)
(0.00,3.00) -- (3.00,3.00)
(0.75,2.25) -- (1.50,2.25)
(0.75,1.875) -- (1.50,1.875)
(0.375,0.00) --(0.375,1.50);
\draw
(1.125,2.05) node {\scalebox{.67}{$g$}}
(1.125,1.65) node {\scalebox{.67}{$c$}};
\end{tikzpicture}
}
\quad
\adjustbox{valign=m}{
\begin{tikzpicture} [ draw = black, x = 6 mm, y = 6 mm ]
\draw
(0.00,0.00) -- (0.00,3.00)
(0.75,0.00) -- (0.75,3.00)
(1.50,0.00) -- (1.50,3.00)
(2.25,0.00) -- (2.25,3.00)
(3.00,0.00) -- (3.00,3.00)
(0.00,0.00) -- (3.00,0.00)
(0.00,1.50) -- (3.00,1.50)
(0.00,3.00) -- (3.00,3.00)
(0.75,2.25) -- (1.50,2.25)
(0.75,1.875) -- (1.50,1.875);
\draw
(1.125,2.05) node {\scalebox{.67}{$g$}}
(1.125,1.65) node {\scalebox{.67}{$c$}};
\end{tikzpicture}
}
\quad
\adjustbox{valign=m}{
\begin{tikzpicture} [ draw = black, x = 6 mm, y = 6 mm ]
\draw
(0.00,0.00) -- (0.00,3.00)
(0.75,0.00) -- (0.75,3.00)
(1.50,0.00) -- (1.50,3.00)
(2.25,0.00) -- (2.25,3.00)
(3.00,0.00) -- (3.00,3.00)
(0.00,0.00) -- (3.00,0.00)
(0.00,1.50) -- (3.00,1.50)
(0.00,3.00) -- (3.00,3.00)
(0.75,2.25) -- (1.50,2.25)
(0.75,0.75) -- (1.50,0.75);
\draw
(1.125,1.875) node {$g$}
(1.125,1.1) node {$c$};
\end{tikzpicture}
}
\end{align*}

\begin{theorem}
In every double interchange semigroup, the following commutativity relation holds
for all values of the arguments $a, \dots, j$:
\[
\begin{array}{l}
( ( a \wedgehor ( b \wedgever c ) ) \wedgever ( f \wedgehor ( g \wedgever h ) ) )
\wedgehor
( ( d \wedgehor e ) \wedgever ( i \wedgehor j ) )
\equiv {}
\\[1mm]
( ( a \wedgehor ( b \wedgever g ) ) \wedgever ( f \wedgehor ( c \wedgever h ) ) )
\wedgehor
( ( d \wedgehor e ) \wedgever ( i \wedgehor j ) )
\end{array}
\]
\end{theorem}

\subsubsection{Configuration $C$}

For configuration $C$ in display \eqref{configsABC}, recall that applying the interchange law
does not change the partition (only the monomial representing the partition), and applying associativity
can be done only horizontally to the entire configuration or vertically to the second slice from the left.
None of these operations transposes the two smallest empty blocks, so we obtain no commutativity
relation.


\subsection{Three parallel horizontal slices}

In this subsection we consider horizontal rather than vertical slices, since this makes it a little easier
to follow the discussion.
We do not claim to have discovered all possible commutativity relations with three parallel slices, since
the number of cases is very large.
However, we determine 32 commutativity relations, 16 of which are new, and 16 of which follow immediately
from one of the known arity nine relations \cite{BM2016}.
Moreover, the 16 new relations may all be obtained from a single relation by applying associativity and
the automorphism group of the square (the dihedral group of order 8).
Without loss of generality, this leaves the following two cases.

\emph{Case 1}:
The horizontal slices have 2, 6, 2 empty blocks, labelled as follows:
\begin{center}
\begin{tikzpicture} [ draw = black, x = 8 mm, y = 8 mm ]
\draw
(0,0) -- (3,0)
(0,1.5) -- (3,1.5)
(0,3) -- (3,3)
(0,2.25) -- (3,2.25)
(0,0) -- (0,3)
(0.75,1.5) -- (0.75,2.25)
(1.5,0) -- (1.5,3)
(1.125,1.5) -- (1.125,2.25)
(2.25,1.5) -- (2.25,2.25)
(1.875,1.5) -- (1.875,2.25)
(3,0) -- (3,3)
(0.75,0.75) node {$a$}
(2.25,0.75) node {$b$}
(0.375,1.875) node {$c$}
(0.9375,1.875) node {$d$}
(1.3125,1.875) node {$e$}
(1.6875,1.875) node {$f$}
(2.0625,1.875) node {$g$}
(2.625,1.875) node {$h$}
(0.75,2.625) node {$i$}
(2.25,2.625) node {$j$};
\end{tikzpicture}
\end{center}
The two commutating empty blocks could be any two of $d$, $e$, $f$, $g$.
But in this configuration, it is easy to see that no sequence of applications of associativity
and the interchange law will change the order of these four blocks.

\emph{Case 2}:
The horizontal slices have 2, 5, 3 empty blocks.
There are two subcases, depending on whether the third horizontal slice has two vertical cuts,
or one vertical cut and one horizontal cut.
In the latter subcase, we label the blocks as follows:
\begin{equation}
\label{oneofeach}
\adjustbox{valign=m}{
\begin{tikzpicture} [ draw = black, x = 8 mm, y = 8 mm ]
\draw
(0,0) -- (3,0)
(0,1.5) -- (3,1.5)
(0,3) -- (3,3)
(0,2.25) -- (3,2.25)
(1.5,0.75) -- (3,0.75)
(0,0) -- (0,3)
(0.75,1.5) -- (0.75,2.25)
(1.5,0) -- (1.5,3)
(2.25,1.5) -- (2.25,2.25)
(1.92,1.5) -- (1.92,2.25)
(3,0) -- (3,3)
(0.75,0.75) node {$a$}
(2.25,0.375) node {$b$}
(2.25,1.125) node {$c$}
(0.375,1.875) node {$d$}
(1.6875,1.875) node {$f$}
(2.0625,1.875) node {$g$}
(2.625,1.875) node {$h$}
(1.125,1.875) node {$e$}
(0.75,2.625) node {$i$}
(2.25,2.625) node {$j$};
\end{tikzpicture}
}
\end{equation}

\begin{theorem}
In every double interchange semigroup, the following commutativity relation holds
for all values of the arguments $a, \dots, j$:
\[
\begin{array}{l}
(a \wedgehor (b \wedgever c))\wedgever ((d \wedgehor e) \wedgever i)
\wedgehor
(((f \wedgehor g) \wedgehor h) \wedgever j)
\equiv {}
\\[1mm]
(a \wedgehor (b \wedgever c))  \wedgever (( d \wedgehor e) \wedgever i)
\wedgehor
(((g  \wedgehor  f)\wedgehor h) \wedgever j)
\end{array}
\]
\end{theorem}

\begin{proof}
We list applications of interchange; the other details are self-explanatory:
\begin{align*}
&
(a \wedgehor (b \wedgever c))\wedgever ((d \wedgehor e) \wedgever i) \wedgehor ((f \wedgehor g \wedgehor h) \wedgever j)
\\
&\equiv
(a \wedgehor (b \wedgever c))\wedgever ((d \wedgehor e \wedgehor f \wedgehor g \wedgehor h)   \wedgever ( i \wedgehor j)
\\
&\equiv
(a \wedgehor (b \wedgever c))\wedgever ((d \wedgehor e \wedgehor f) \wedgever i) \wedgehor ((g \wedgehor h) \wedgever j)
\\
&\equiv
(a \wedgever (d \wedgehor e \wedgehor f) \wedgever i) ) \wedgehor ((b \wedgever c) \wedgever ((g \wedgehor h) \wedgever j)
\\
&\equiv
((a \wedgever (d \wedgehor e \wedgehor f)) \wedgehor (b \wedgever c) ) ) \wedgever (i \wedgehor ((g \wedgehor h) \wedgever j)
\\
&\equiv
((a \wedgehor b ) \wedgever ( (d \wedgehor e \wedgehor f \wedgehor c) ) \wedgever (i \wedgehor ((g \wedgehor h) \wedgever j)
\\
&\equiv
((a \wedgever  (d \wedgehor e)) \wedgehor (b \wedgever (f \wedgehor c) ) \wedgever (i \wedgehor ((g \wedgehor h) \wedgever j)
\\
&\equiv
((a \wedgever  (d \wedgehor e)) \wedgever i ) \wedgehor  ( (b \wedgever (f \wedgehor c)  \wedgever ((g \wedgehor h) \wedgever j)
\\
&\equiv
(a \wedgehor  (b \wedgever (f \wedgehor c))  \wedgever ((d \wedgehor e) \wedgever i) \wedgehor ((g \wedgehor h) \wedgever j))
\\
&\equiv
(a \wedgehor  (b \wedgever (f \wedgehor c))  \wedgever (d \wedgehor e \wedgehor g \wedgehor h) \wedgever (i \wedgehor j))
\\
&\equiv
((a \wedgever  (d \wedgehor e \wedgehor g ))  \wedgehor (b \wedgever (f \wedgehor c) \wedgever h)) \wedgever (i \wedgehor j)
\\
&\equiv
(a \wedgever  (d \wedgehor e \wedgehor g ) \wedgever i) \wedgehor (b \wedgever (f \wedgehor c) \wedgever h \wedgever j)
\\
&\equiv
(a \wedgehor  b) \wedgever ((d \wedgehor e \wedgehor g ) \wedgever i) \wedgehor ((f \wedgehor c) \wedgever h \wedgever j))
\\
&\equiv
(a \wedgehor  b) \wedgever (d \wedgehor e \wedgehor g  \wedgehor  f \wedgehor c) \wedgever ( i \wedgehor (h \wedgever j))
\\
&\equiv
((a \wedgever (d \wedgehor e \wedgehor g  \wedgehor  f)  \wedgehor (b \wedgever c)) \wedgever ( i \wedgehor (h \wedgever j))
\\
&\equiv
(a \wedgever (d \wedgehor e \wedgehor g  \wedgehor  f ) \wedgever  i ) \wedgehor ((b \wedgever c) \wedgever (h \wedgever j))
\\
&\equiv
(a \wedgehor (b \wedgever c))  \wedgever ((( d \wedgehor e \wedgehor g  \wedgehor  f)\wedgever i) \wedgehor (h \wedgever j))
\\
&\equiv
(a \wedgehor (b \wedgever c))  \wedgever (( d \wedgehor e \wedgehor g  \wedgehor  f\wedgehor h) \wedgever (i \wedgehor j))
\\
&\equiv
(a \wedgehor (b \wedgever c))  \wedgever (( d \wedgehor e) \wedgever i) \wedgehor ((g  \wedgehor  f\wedgehor h) \wedgever j)
\end{align*}
The proof is complete.
\end{proof}

In the subcase with two vertical cuts in the third slice, the corresponding diagram is the same as \eqref{oneofeach}
except that the lower right block containing $b$ and $c$ is rotated 90 degrees clockwise.
We obtain a commutativity relation for this block partition, but this relation is easily seen to be an consequence
of identity 3992 from \cite{BM2016}.


\section{Concluding remarks}

In this final section we briefly mention possible directions for future research.

\subsubsection*{Mixed structures}

We have studied two binary operations, both associative or both nonassociative,
related by the interchange law.
More generally, for $p, q \ge 0$ let $I = \{ 1, \dots, p{+}q \}$;
choose subsets $J \subseteq I$ and
$K \subseteq \{ \, \{k,\ell\} \mid k, \ell \in I, \, k \ne \ell \, \}$.
Let $S$ be a set with $p{+}q$ binary operations,
$p$ associative $\star_1$, \dots, $\star_p$ and $q$ nonassociative $\star_{p+1}$, \dots, $\star_{p+q}$,
satisfying interchange between $\star_j$ and itself for $j \in J$,
and between $\star_k$, $\star_\ell$ for $\{k,\ell\} \in K$.
The operads we have studied in this paper correspond to
$(p,q) = (0,2)$ or $(2,0)$ with $J = \emptyset$ and $K = \{\{1,2\}\}$.

\subsubsection*{Higher arity interchange laws}

We have studied only binary operations.
More generally, let $S$ be a nonempty set,
$M_p(S)$ the set of all $p$-ary operations $f\colon S^p \to S$, and
$X = ( x_{ij} )$ a $p \times q$ array with entries in $S$.
If $f \in M_p(S)$, $g \in M_q(S)$ then we may apply $f$, $g$ to $X$
either by applying $g$ to each row vector, obtaining an $m \times 1$ column vector, and applying $f$;
or the reverse.
If the results are equal then $f$, $g$ satisfy the $m \times n$ interchange law
(we also say that $f$, $g$ commute):
\[
\begin{array}{l}
f( g( x_{11}, \dots, x_{1n} ), \dots, g( x_{m1}, \dots, x_{mn} ) ) \equiv
\\
g( f( x_{11}, \dots, x_{m1} ), \dots, f( x_{1n}, \dots, x_{mn} ) ).
\end{array}
\]
Since $f$ acts on columns and $g$ on rows, we may write $f( X g ) \equiv ( f X ) g$,
showing that interchange may be regarded as a form of associativity.

\subsubsection*{Higher dimensions}

We have studied structures with two operations, corresponding to the horizontal
and vertical directions in two dimensions.
Most of our constructions make sense for any number of dimensions $d \ge 2$.
One obstacle for $d \ge 3$ is that the monomial basis for $\mathbf{Assoc}$ consisting of
nonbinary trees with alternating white and black internal nodes ($\mathbf{AssocNB}$)
does not generalize in a straightforward way.

\subsubsection*{Associativity for two operations}

With more than one operation, there are various forms of associativity;
we have only considered the simplest: each operation is individually associative.
The operations may also associate with each other in various ways:
black-white associativity, $( a \wedgehor b ) \wedgever c \equiv a \wedgehor ( b \wedgever c )$;
total associativity (black-white and white-black); compatibility
(every linear combination of the operations is associative);
diassociativity (black-white and the two bar identities).

\subsubsection*{Variations on the interchange law}

In universal algebra, the interchange law is called the \emph{medial} identity;
it has a close relative, the \emph{paramedial} identity, in which the outer arguments transpose:
$( a \wedgehor b ) \wedgever ( c \wedgehor d ) = ( d \wedgever b ) \wedgehor ( d \wedgever a )$.
In general, one considers $d$ operations of arity $n$, and relations $m_1 \equiv m_2$
where $\{ m_1, m_2 \}$ is an unordered pair of monomials of arity $N = 1+w(n{-}1)$
in which $m_1$ has the identity permutation of $N$ distinct variables and $m_2$ has some
nonidentity permutation.
Of greatest interest are those relations which have the greatest symmetry:
that is, the corresponding unordered pair generates an orbit of minimal size
under the action of the wreath product $S_d \ltimes (S_n)^d$ of the symmetric group $S_d$
permuting the operations with the group $(S_n)^d$ permuting the arguments of the operations.

\subsubsection*{$N$-ary suboperads of binary operads}

To conclude, we mention a different point of view on commutativity for double interchange semigroups.
In general, let $\mathbf{O}$ be a symmetric operad generated by binary operations satisfying
relations of arity $\ge 3$.
An algebra over $\mathbf{O}$ is called an $\mathbf{O}$-\emph{algebra};
the most familiar cases are associative, alternative, pre-Lie, diassociative, dendriform, etc.
We propose the following definition of $N$-\emph{tuple} $\mathbf{O}$-\emph{system} for all $N \ge 3$:
an algebra over the suboperad $\mathbf{O}^{(N)} \subset \mathbf{O}$ generated by
the $S_N$-module $\mathbf{O}(N)$ of all $N$-ary operations in $\mathbf{O}$.
In particular, consider the operad $\mathbf{DIA}$ generated by two associative operations
satisfying the interchange law.
Previous results \cite{BM2016} show that $\mathbf{DIA}(N)$ is a direct sum of copies of the
regular $S_N$-module if and only if $N \le 8$.
The generators of $\mathbf{DIA}$ have no symmetry,
but the generators of $\mathbf{DIA}^{(N)}$ have symmetry for $N \ge 9$.




\end{document}